\documentclass[11pt,twoside]{amsart}
\usepackage{bbm}
\usepackage{enumerate}
\usepackage{graphicx,color}
\usepackage[all]{xy}
\usepackage{relsize}
\usepackage[utf8]{inputenc}
\usepackage{verbatim}
\usepackage[dvipsnames]{xcolor}
\usepackage{amsmath}

\theoremstyle{plain}
\newtheorem{The}{Theorem}
\newtheorem*{The*}{Theorem}
\newtheorem{Pro}{Proposition}
\newtheorem{Lem}{Lemma}

\newtheorem*{Cor*}{Corollary}

\theoremstyle{definition}
\newtheorem{Def}{Definition}
\newtheorem{Rem}{Remark}
\newtheorem{Exa}{Example}
\newtheorem*{Rem*}{Remark}

\numberwithin{equation}{section}
\numberwithin{The}{section}
\numberwithin{Def}{section}
\numberwithin{Rem}{section}
\numberwithin{Exa}{section}
\numberwithin{Lem}{section}
\numberwithin{Pro}{section}
\numberwithin{Cor}{section}

\DeclareMathOperator{\Tr}{tr}             

\renewcommand{\Im}{\operatorname{Im}}
\renewcommand{\Re}{\operatorname{Re}}
\newcommand{\II}{\operatorname{II}}

\DeclareMathOperator{\del}{\partial}

\newcommand{\bigletter}[1]{\mathlarger{#1}}

\newcommand{\bigPi}{\bigletter{\Pi}}
\newcommand{\R}{\mathbb{R}}

\newcommand{\C}{\mathbb{C}}
\newcommand{\N}{\mathbb{N}}
\newcommand{\Z}{\mathbb{Z}}

\renewcommand{\H}{\mathbb{H}}  

\newcommand{\CP}{\mathbb{CP}}

\begin{document}

\title{Candidates for non-rectangular constrained Willmore minimizers}

\author{Lynn Heller}

\address{ Institut f\"ur Differentialgeometrie\\  Leibniz Universit{\"a}t Hannover\\ Welfengarten
1\\ 30167 Hannover\\ Germany}
 
 \email{lynn.heller@math.uni-hannover.de}
 \author{Cheikh Birahim Ndiaye}

\address{Department of Mathematics of Howard University\\
204 Academic Support Building B
Washington, DC 20059k \\ USA
 }
 \email{cheikh.ndiaye@howard.edu}

 \date{\today}

\thanks{The first  author wants to thank the Ministry of Science, Research and Art Baden-W\"uttemberg and the European social Fund for supporting her research within the Margerete von Wrangell Programm. Further, both authors are indebted to the Baden-W\"urttemberg foundation for supporting the project within the Eliteprogramm for Postdocs. Moreover, the second author thank the Deutsche Forschungsgemeinschaft (DFG) for financial support through the project ''Fourth-order uniformization type theorem for $4$-dimensional Riemannian manifolds'' and the Schweizerischer Nationalfonds (SNF) for financial support through the project PP00P2\_144669}
\noindent

\maketitle

\begin{abstract} 
For every $\;b>1\;$ fixed, we explicitly construct $1$-dimensional families of embedded constrained Willmore tori parametrized by their conformal class $\;(a,b)$\; with  $\; a \sim_b 0^+\;$ deforming the homogenous torus \;$f^b$ of conformal class \;$(0,b).$ The variational vector field  at $f^b$ is hereby given by a non-trivial 
zero direction of a penalized Willmore stability operator which we show to coincide with a double point of the corresponding spectral curve. Further, we characterize for $b \sim 1$, $b \neq 1$ and $a \sim_b 0^+$ the family obtained by opening the ``smallest" double point on the spectral curve which is heuristically the direction with the smallest increase of Willmore energy at $f^b$. Indeed we show in \cite{HelNdi1} that these candidates minimize the Willmore energy in their respective conformal class for $b \sim 1$, $b \neq 1$ and $a \sim_b 0^+.$
 \begin{center}
\vspace{0.5cm}
\includegraphics[width= 0.19\textwidth]{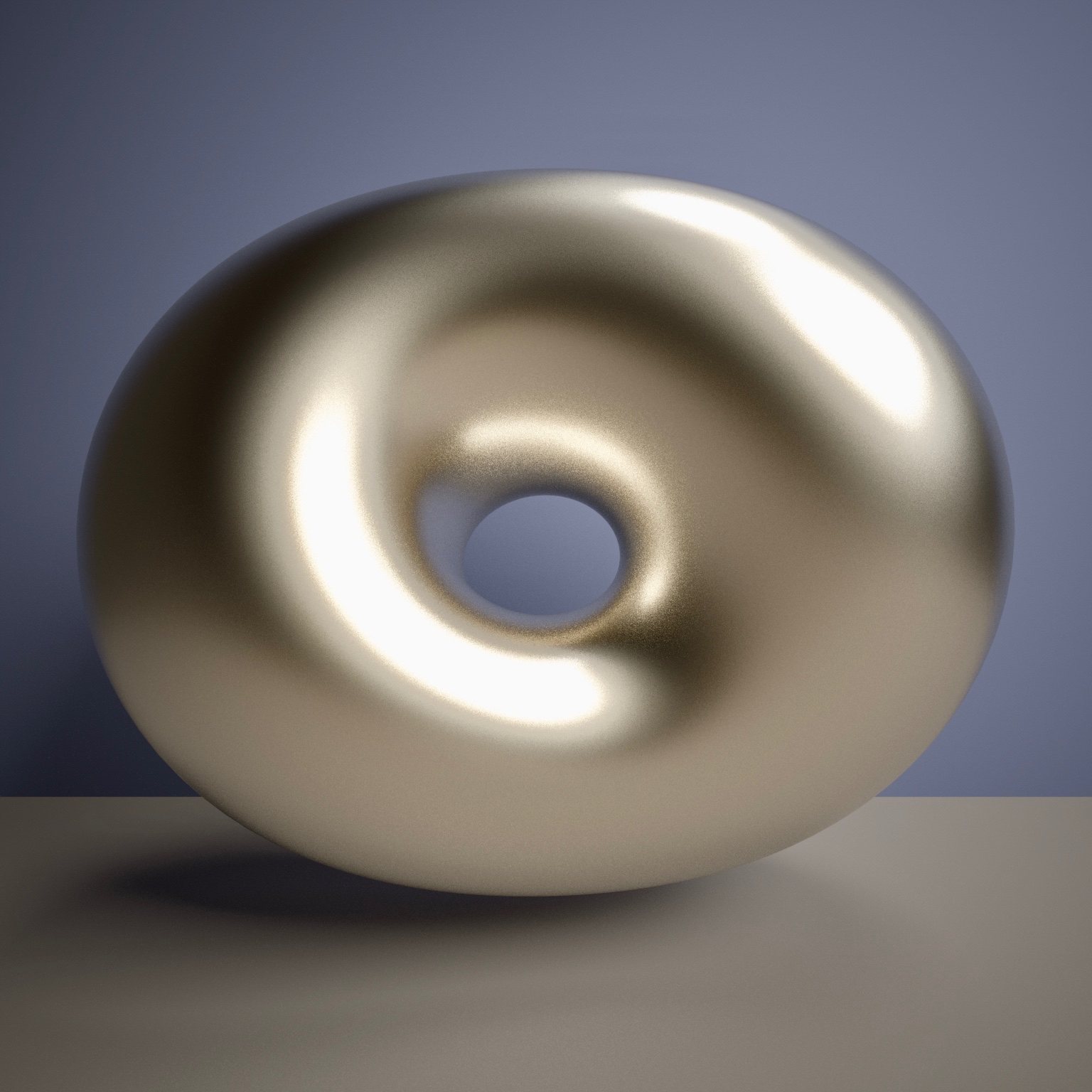}\hspace{0.2cm}
\includegraphics[width= 0.19\textwidth]{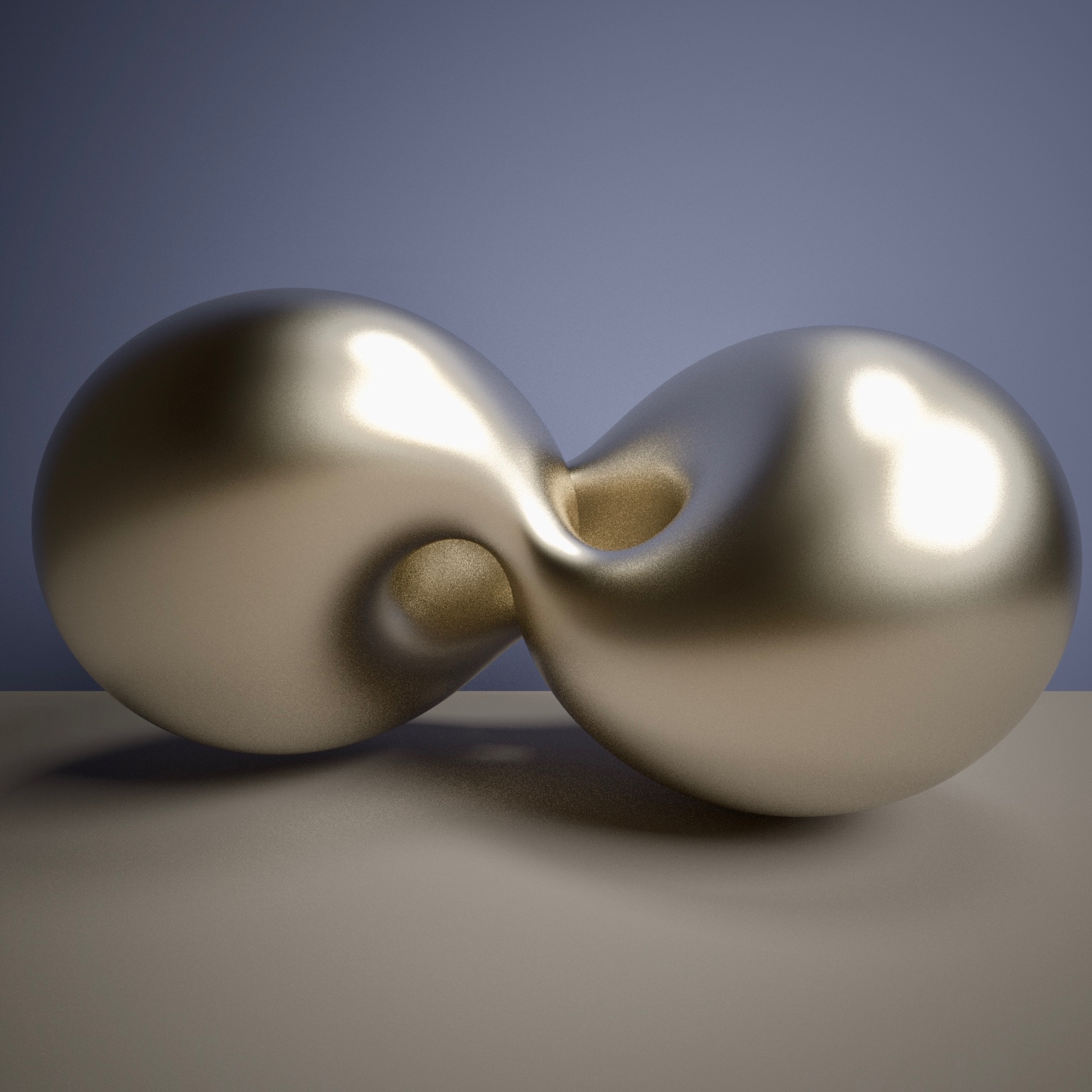}\hspace{0.2cm}
\includegraphics[width= 0.19\textwidth]{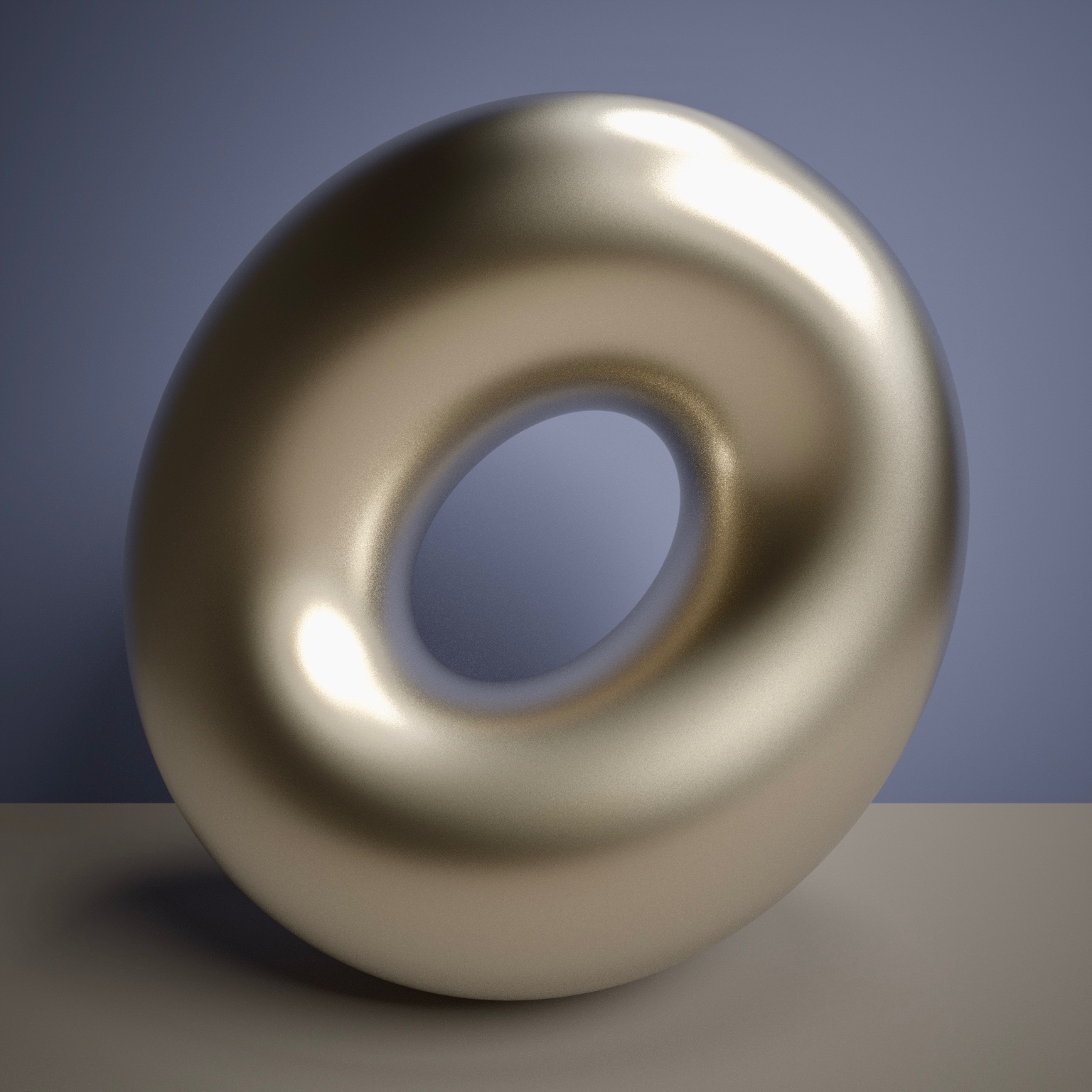}\hspace{0.2cm}
\includegraphics[width= 0.19\textwidth]{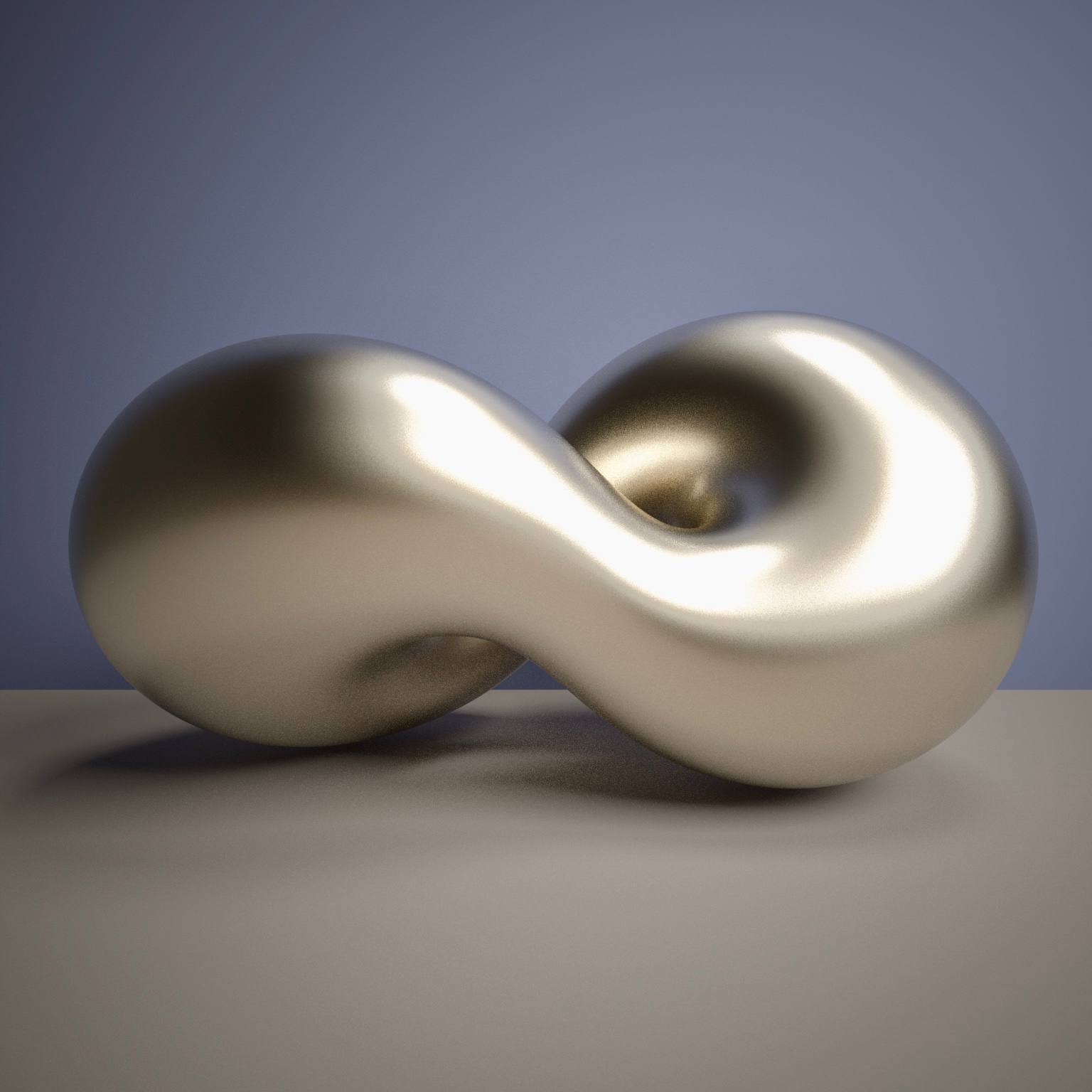}
\vspace{0.6cm}
\end{center}

 \end{abstract}




\section{Introduction and statement of the results}
In the 1960s Willmore \cite{Willmore}  proposed to study the critical values and critical points of the {\em bending energy} 
\[
\mathcal{W}(f)=\int_M H^2dA,
\]

the average value of the squared mean curvature $\;H\;$ of an immersion $\;f\colon M\longrightarrow \R^3\;$ of a closed surface  $\;M.\;$ In this definition we denote by $\;dA\;$ the induced volume form and  $\;H:= \tfrac{1}{2}tr (\text{II})\;$ with II the second fundamental form of the immersion $\;f.\;$
Willmore showed that the absolute minimum of this functional is attained at round spheres with Willmore energy $\;\mathcal{W}=4\pi.\;$ He also conjectured that the minimum over surfaces of genus $1$ is attained at (a suitable stereographic projection of) the Clifford torus in the 3-sphere with $\;\mathcal{W}=2\pi^2.\;$
It soon was noticed that the bending energy $\;\mathcal{W}\;$ (by then also known as the {\em Willmore energy}) is invariant under M\"obius transformations of the target space -- in fact, it is invariant under conformal changes of the metric in the target space, see \cite{Blaschke, Chen}. Thus, it makes no difference for the study of the Willmore functional which constant curvature  target space is chosen. \\

Bryant \cite{Bryant} characterized all Willmore spheres as M\"obius transformations of genus $\;0\;$ minimal surfaces in $\;\R^3\;$ with planar ends. The value of the bending energy on Willmore spheres is thus quantized to be $\; \mathcal{W}=4\pi k,\;$ with $\;k\geq 1\;$ the number of ends. With the exception of  $\;k=2,3,5,7\;$ all values occur. For more general target spaces the variational setup to study this surfaces can be found in \cite{Mondino}.
The first examples of Willmore surfaces not M\"obius equivalent to minimal surfaces were found by Pinkall \cite{Pinkall}. They were constructed via lifting elastic curves $\;\gamma\;$ with geodesic curvature $\;\kappa\;$ on the 2-sphere under the Hopf fibration to Willmore tori in the 3-sphere, where
elastic curves are the critical points for the elastic energy 

$$E(\gamma)= \int_{\gamma} (\kappa^2+1)ds$$

 and $\;s\;$ is the arclength parameter of the curve. Later Ferus and Pedit \cite{FerusPedit} classified all Willmore tori equivariant under a M\"obius $\;S^1$-action on the 3-sphere (for the definition of $\;S^1$-action see Definition \ref{S1action}).\\

The Euler-Lagrange equation for the Willmore functional 

\[
{\Delta} H+ 2H(H^2-K)=0,
\]

 where $\;K\;$ denotes the Gau\ss ian curvature of the surface $\;f\colon M\longrightarrow \R^3\;$ and $\;\Delta\;$ its Laplace-Beltrami operator, is a 4th order elliptic PDE for $\;f\;$ since the mean curvature vector $\;\vec{H}\;$ is the normal part of $\;\Delta f.\;$ Its analytic properties are prototypical for non-linear bi-Laplace equations. Existence of a minimizer for the Willmore functional $\;\mathcal{W}\;$ on the space of smooth immersions from 2-tori was shown by Simon \cite{Simon}. Bauer and Kuwert \cite{BauerKuwert} generalized this result to higher genus surfaces. After a number of partial results, e.g. \cite{LiYau}, \cite{MontielRos}, \cite{Ros}, \cite{Top}, \cite{FLPP},  Marques and Neves  \cite{MarquesNeves}, using Almgren-Pitts min-max theory, gave a proof of the Willmore conjecture in $3$-space in 2012. An alternate strategy was proposed in \cite{Schmidt}\\

A more refined, and also richer, picture emerges when restricting the Willmore functional to the subspace of smooth immersions $\;f\colon M\longrightarrow \R^3\;$ inducing a given conformal structure on $\;M.\;$  Thus, $\;M\;$  now is  a Riemann surface and we study the Willmore energy $\;\mathcal{W}\;$ on the space of smooth conformal immersions $\;f\colon M\longrightarrow \R^3\;$ whose critical points are called {\em (conformally) constrained Willmore surfaces}. The conformal constraint augments the Euler-Lagrange equation by  $\;\omega\in H^0(K^2_M)\;$ paired with the trace-free second fundamental  form $\;\mathring{\text{II}}\;$ of the immersion

\begin{equation}\label{CWEL}
\Delta H+ 2H(H^2-K)=\,<\omega,\mathring{\text{II}}> ,
\end{equation}

where $\;H^0(K^2_M)\;$ denoting the space of holomorphic quadratic differentials. In the Geometric Analytic literature, the space $\;H^0(K_M^2)\;$ is also referred to as $\;S_2^{TT}\big(g_{euc}\big)\;$ the space of symmetric, covariant, transverse and traceless $2$-tensors with respect to the euclidean metric $\;g_{euc},\;$ In the case of tori, the Teichm\"ller space can be identified with the upper halfplane $\H^2$. With $\;\bigPi = (\bigPi^1, \bigPi^2)\;$  denoting the projection map from the space of immersions to $\H^2$,  the right hand side of the Euler Lagrange equation can be written as
$$<\omega,\mathring{\text{II}}> = \alpha \delta \bigPi^1 + \beta\delta \bigPi^1,$$
for real numbers $\alpha$ and $\beta$, which are referred to as the $\bigPi$-Lagrange multipliers of the immersion.

Since there are no holomorphic (quadratic) differentials on a genus zero 
Riemann surface, constrained Willmore spheres are the same as Willmore spheres. For higher genus surfaces this is no longer the case: constant mean curvature surfaces (and their M\"obius transforms) are constrained Willmore, as one can see by choosing $\;\omega:=\mathring{\text{II}}\;$ as the holomorphic Hopf differential in the Euler Lagrange equation \eqref{CWEL}, but not Willmore unless they are minimal in a space form. 
Bohle \cite{Bohle}, using techniques developed in \cite{BohLesPedPin} and \cite{BohPedPin_ana}, showed that all constrained Willmore tori have finite genus spectral curves and are described by linear flows on the Jacobians of those spectral curves\footnote{For the notion of spectral curves and the induced linear flows on the Jacobians see \cite{BohLesPedPin}.}. Thus the complexity of the map $\;f\;$ heavily depends on the genus  its spectral curve $\;\Sigma$ -- the spectral genus -- giving the dimension of the Jacobian of $\;\Sigma\;$ and thus codimension of the linear flow. The simplest examples of constrained Willmore tori, which have spectral genus zero, are the tori of revolution in $\;\R^3\;$ with circular profiles -- the homogenous tori. Those are stereographic images of products of circles of varying radii ratios in the 3-sphere and thus have constant mean curvature as surfaces in the 3-sphere.  Starting at the Clifford torus, which has mean curvature $\;H=0\;$ and a square conformal structure, these homogenous tori in the 3-sphere parametrized by their mean curvature $\;H\;$ ``converge'' to a circle as $\;H\longrightarrow \infty\;$ and thereby sweeping out  all rectangular conformal structures. Less trivial examples of constrained Willmore tori come from the Delaunay tori of various lobe counts (the $\;n$-lobed Delaunay tori) in the 3-sphere whose spectral curves have genus $1$, see Figure~\ref{fig:torus-tree} and \cite{KilianSchmidtSchmitt1} for their definition. \\
\begin{figure}
\includegraphics[width=0.5\textwidth]{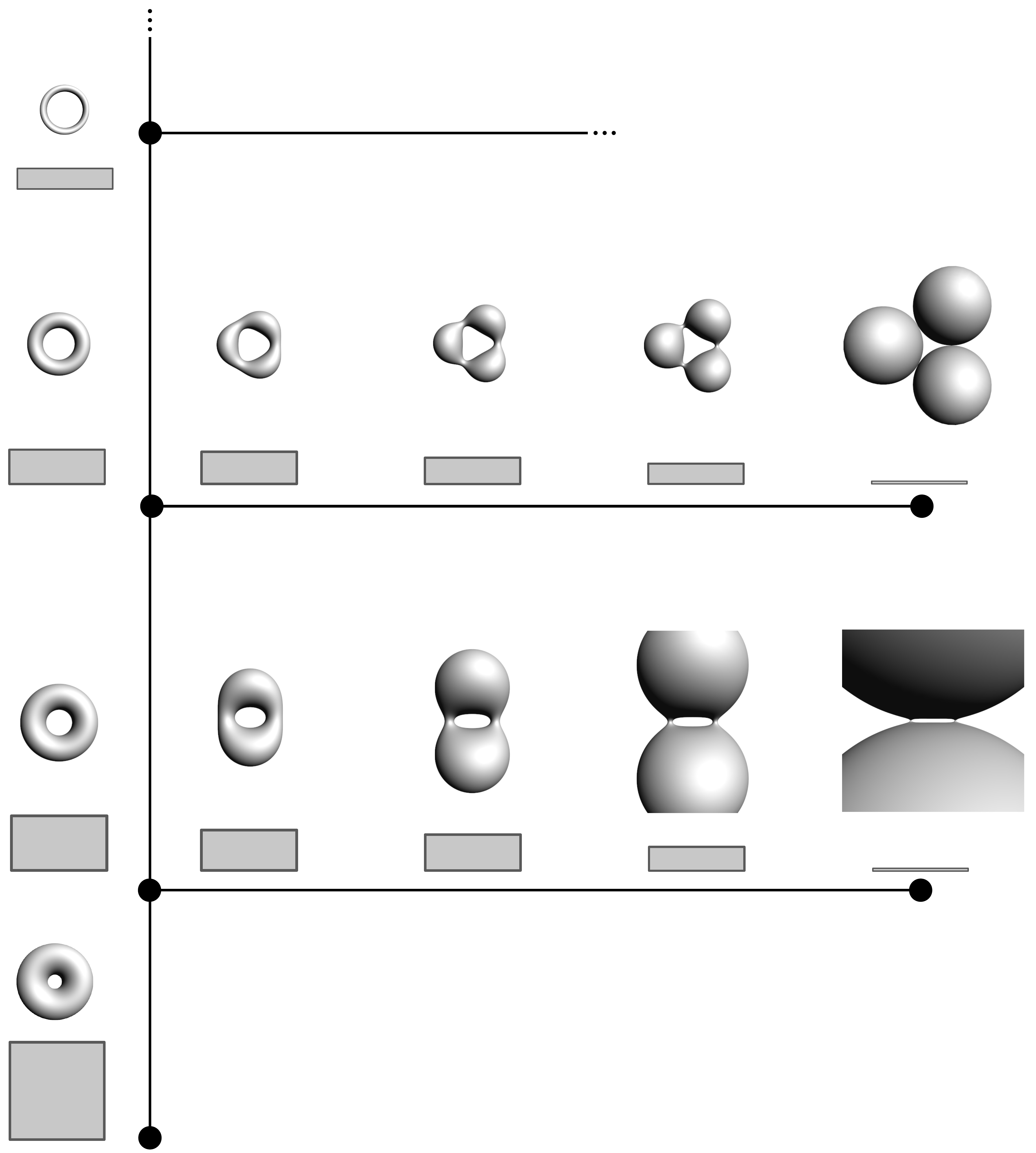}
\caption{
The vertical stalk represents the family of homogenous tori,
starting with the Clifford torus at the bottom.
Along this stalk are bifurcation points from which embedded Delaunay tori continue the homogenous family. The rectangles indicate the conformal types. The family of surfaces starting at the Clifford torus, bifurcating at the first branch point has Willmore energy below $\;8 \pi\;$ and is conjectured to be the minimizer in their respective conformal class.
 Image by Nicholas Schmitt.
}
\label{fig:torus-tree}
\end{figure}

Existence and regularity of a $\;W^{2,2} \cap W^{1, \infty}\;$ minimizer $\;f\colon M\longrightarrow \R^3\;$ for a prescribed Riemann surface structure\footnote{For the notion of $\;W^{2,2} \cap W^{1, \infty}\;$ immersions see \cite{KuwertSchaetzle}, \cite{Riviere} or \cite{KuwertLi}.} (constrained Willmore minimizer) was shown by \cite{KuwertSchaetzle}, \cite{KuwertLi}, \cite{Riviere2} and \cite{Schaetzle} under the assumption that the infimum Willmore energy in the conformal class  is below $\;8\pi.\;$  The latter assumption ensures that minimizers are embedded by the Li and Yau inequality \cite{LiYau}. A broader review of analytic results for Willmore surfaces can be found in the lecture notes \cite{KuwertSchaetzle} and \cite{Riviere3}, see also the references therein. \\

Ndiaye and Sch\"atzle \cite{NdiayeSchaetzle1, NdiayeSchaetzle2} identified the first explicit constrained Willmore minimizers (in every codimension) for rectangular conformal classes in a neighborhood (with size depending on the codimension) of the square class to be the homogenous tori.  These tori of revolution with circular profiles, whose spectral curves have genus 0, eventually have to fail to be minimizing in their conformal class for $\;H >>1,\;$ since their Willmore energy can be made arbitrarily large and every rectangular torus can be conformally embedded into $\;\R^3\;$ (or $\;S^3$) with Willmore energy below $\;8\pi,\;$ see \cite{ KilianSchmidtSchmitt1, NdiayeSchaetzle2}. Calculating the 2nd variation of the Willmore energy $\;\mathcal{W}\;$ along homogenous tori Kuwert and Lorenz \cite{KuwertLorenz} showed that zero eigenvalues only appear at those conformal classes whose rectangles have side length ratio $\;\sqrt{k^2-1}\;$ for an integer $\;k\geq 2,\;$ at which the index of the surface increase. These are exactly the rectangular conformal classes from which the $\;k$-lobed Delaunay tori (of spectral genus 1) bifurcate.  Any of the families starting from the Clifford torus, following homogenous tori to the $\;k$-th  bifurcation point, and continuing with the $\;k$-lobed Delaunay tori sweeping out all rectangular classes (see Figure~\ref{fig:torus-tree}) ``converge'' to a neckless of spheres as conformal structure degenerates. The Willmore energy $\;\mathcal W\;$ of the resulting family\footnote{For simplicity we call this family in the following the $\;k$-lobed Delaunay tori.} is strictly monotone and satisfies $\;2\pi^2\leq \mathcal{W}<4\pi k,\;$ see \cite{KilianSchmidtSchmitt1, KilianSchmidtSchmitt2}.  
Thus for $\;k=2\;$ the existence of $2$-lobed Delaunay tori imply that the infimum Willmore energy in every rectangular conformal class is always below $\;8 \pi\;$ and hence there exist embedded constrained Willmore minimizers for these conformal types by \cite{KuwertSchaetzle}.
It is conjectured that the minimizers for $\;\mathcal{W}\;$ in rectangular conformal classes are given by the $2$-lobed Delaunay tori. For a more detailed discussion of the $2$-lobe-conjecture see \cite{HePe}.  Surfaces of revolution with prescribed boundary values was studied in \cite{Grunau}. \\

In this paper we turn our attention to finding families of constrained Willmore tori deforming homogenous tori parametrized by their conformal class. Moreover, we identify a family of putative constrained Willmore minimizers in non-rectangular conformal classes. These candidates are shown in \cite{HelNdi1} to be actually minimizing when restricting to $a^b \sim_b 0^+$ and $b \sim 1, $ but $b \neq1.$ Our main theorem is the following.

 \begin{The} [Main Theorem]\label{explicitcandidates}
For every $\;b \sim 1,\;$ $b \neq1\;$ fixed, there exist a family  of constrained Willmore tori $\;f_{(a, b)}\;$ parametrized by their conformal class $\;(a,b),\;$ with $a \sim_b 0^+$ satisfying the following properties:
\begin{itemize}
\item $f_{(0,b)} = f^b\;$ is homogenous and parametrized as a $(1,2)$-equivariant surface 
\begin{equation}\label{12fb}
\begin{split}
f^b \colon T^2_b : = \C/\left (i 2\pi \Z \oplus 2\pi\frac{ rs + i 2r^2}{s^2+ 4r^2}\Z\right) \longrightarrow S^3\\
(x,y) \longmapsto \left(r e^{i\left(y- 2\tfrac{r}{s}x\right)}, s e^{i\left(2y + \tfrac{s}{r}x \right)}\right)
\end{split}
\end{equation}
\item $f_{(a,b)}\;$ is non degenerate for $\;a \neq0,\;$ and $\;f_{(a,b)} \longrightarrow f^b\;$ smoothly as $\;a \longrightarrow 0$, 
\item  for every $\;b \sim 1,\;$ $b \neq 1\;$ fixed the normal variation $\;\big(\del_{\sqrt{a}}f_{(a,b)} |_{a=0}\big)^\perp\;$ verifies $$<\partial_{\sqrt{a}} f_{(a,b)}\mathlarger{|}_{a=0}, \vec{n}_{1,2}^{b} > \;=  \sin \left (\tfrac{s}{r}+  4\tfrac{r}{s} x \right),$$ 
where $\vec{n}_{1,2}^b$ is the normal vector of $f^b$ in the $(1,2)$-parametrization.
\item for every $\;b \sim 1,\;$ $b \neq 1\;$ fixed and $\;a\neq 0,\;$ the corresponding $\bigPi$-Lagrange multipliers $\;\alpha_{(a,b)},\;$ and $\;\beta_{(a,b)}\;$ satisfy

$$\alpha_{(a,b)} \nearrow \hat\alpha^b \quad \text{ and } \quad  \beta_{(a,b)}\longrightarrow \beta^b, \quad \text{ as } a \longrightarrow 0$$
for some \;$\hat\alpha^b>0$ \;and \;$\beta^b$ is the $\bigPi^1$-Lagrange multiplier of $f^b.$
\end{itemize}
\end{The}

The surfaces constructed are automatically embedded as they converge smoothly to the homogenous tori $f^b$. In fact we construct multiple families of embedded constrained Willmore tori deforming \;$f^b$ (In the Hopf tori case, we obtain a family for every $b>1$).  The family stated in the Theorem is the one with the smallest energy close to $f^b$ among the families constructed. 

\begin{Rem}
Constrained Willmore minimizers in a particular conformal class only exist, if infimum energy in that class is below $8 \pi$ \cite{KuwertSchaetzle}. Therefore, the constructed families helps to obtain a quantitative estimate on the size of the neighborhood of rectangular conformal classes in which constrained Willmore minimizers exist.
\end{Rem}
We want to give some heuristics why this family should be the constrained Willmore minimizer in their respective conformal class. 
\begin{Def}
For $\;\alpha, \beta \in \R\;$ we use the abbreviations

\begin{equation}
\begin{split}
\mathcal W_{\alpha, \beta}(f) &:= \mathcal W(f) - \alpha \bigPi^1(f) - \beta \bigPi^2(f)\\
\mathcal W_{\alpha}(f) &:= \mathcal W(f) - \alpha \bigPi^1(f). \\
\end{split}
\end{equation}

\end{Def}
The main observation is the following.
\begin{Lem}{\cite[Lemma 2.1]{HelNdi1}}\label{mainobservation}$\ $\\
Let $\;\tilde f^{(a,b)}\;$ be a family of smooth immersions with conformal type 

$${(a,b) =: (\tilde a^2, b) \in [0, a^2_0) \times (1-b_0, 1+b_0)}$$ 

for some positive numbers  $\;a_0,\; b_0 \in \R\;$ such that the map

$$(\tilde a,b) \longmapsto \tilde f^{(a,b)} \in C^2\left(\left[0, a_0\right) \times (1-b_0, 1+b_0), W^{4,2}\right),$$

and  $\;\delta  \bigPi^1 \big({\tilde f^{(0,b)}}\big ) = 0,\;$ but $\;\delta   \bigPi^1 \big({\tilde f^{( a,b)}}\big ) \neq 0\;$ for $\;a \neq 0.\;$  Further, let $\;\tilde \alpha( a,b)\;$ and $\;\tilde \beta( a,b)\;$ be the corresponding Lagrange multipliers  satisfying 

$$(\tilde a,b) \longmapsto \tilde \alpha(a,b), \tilde \beta(a,b) \in C^2\left([0, a_0) \times (1-b_0, 1+b_0), W^{4,2}\right),$$

and 
$\;\tilde \omega( a, b):= \mathcal W\big({\tilde f^{( a,b)}}\big ).\;$ Then we obtain

\begin{enumerate}
\item$$\frac{\partial \tilde \omega( a,b)}{\partial a} = \tilde \alpha(a,b) \text{ for } a \neq 0 \quad \text{ and } \quad \lim_{\tilde a\rightarrow 0} \frac{\partial \tilde \omega(a,b)}{\partial  a}  = \tilde\alpha(0,b) =: \tilde\alpha^b \quad \forall b,$$

\item$$\frac{\partial \tilde \omega( a,b)}{\partial b} = \tilde\beta( a,b) \text{ for } a\neq 0 \quad \text{ and } \quad \lim_{ a\rightarrow 0} \frac{\partial \tilde \omega( a,b)}{\partial b}  = \tilde\beta(0,b)  =: \tilde \beta^b\quad \forall b,$$

\item
$\varphi^b:= \del_{\tilde a} f^{( a, b)}|_{a=0}$ satisfies 

$$ \delta^2\left(\mathcal W_{\tilde \alpha^b, \tilde \beta^b} \right)\big({\tilde f^{( 0,b)}}\big )(\varphi^b, \varphi^b) = 0 \quad \forall b.$$

\end{enumerate}
\end{Lem}

The proof of the Lemma is very straight forward and shows that by fixing $b$,  the variational vector field of a family of constrained Willmore tori $\tilde f^{(a,b)}$, parametrized by its conformal type, at the homogenous torus $f^b$ must be a zero direction of the the second variation of $\mathcal W_{\tilde \alpha^b, \beta^b}$, where $\beta^b$ is the Lagrange multiplier of $f^b.$ Since the Lagrange multipliers are the derivative of the Willmore energy along the family,  the necessary condition to construct a family of minimizers is  that the corresponding Lagrange multipliers converge for $a\longrightarrow 0$ to the smallest possible $\alpha>0$  such that  $\mathcal W_{\alpha, \beta^b}$ has a non-trivial zero direction. Since $\mathcal W_{\alpha=0, \beta^b}$ is strictly stable (invariance), the smallest possible $\alpha$ can be defined as follows.
\begin{Def}\label{alphab}
Let $\;\beta^b\;$ be the $\;\bigPi^2$-Lagrange multiplier of the homogenous torus $\;f^b.\;$ Then we define
$$\alpha^b:=\text{ max }\{\alpha \ | \ \  \delta^2 \mathcal W_{\alpha, \beta^b} \geq0\}.$$
\end{Def}

In \cite[Section 3]{HelNdi1} we computed $\alpha^b$ for $b \sim1$ which is uniquely determined by its non-trivial kernel element. The identification of the normal variation of $f_{(a,b)}$ at $f^b$ in the main theorem therefore shows that the family of constrained Willmore tori constructed here has the desired limit Lagrange multiplier $\hat \alpha^b = \alpha^b$. We have included the Stability computations fo \cite[Section 3]{HelNdi1} in the appendix to make the paper on hand more self-contained.

In this paper we use integrable system theory to construct the families of constrained Willmore tori by opening double points on the spectral curve of the homogenous torus $f^b$.  The resulting surfaces are of spectral genus 2 as they lie in the associated family of constrained Willmore Hopf cylinders \cite[Theorem 9]{He2}. These double points occurs exactly at those points where  $\delta^2 \mathcal W_{\alpha^b, \beta^b}$ has non-trivial zero directions, indicating that the genus of the spectral curve is closely related to the Nullity of the constrained Willmore torus. Due to the different normalization of the Lagrange multipliers in the integrable systems and analysis approach, we decided to compute the normal variations of $f_{(a,b)}$ at $f^b$ in order to identify of the limit Lagrange multiplier.
Moreover, we show that the $\bigPi^1$-Lagrange multiplier converges from below to $\alpha^b$\; as $a \longrightarrow 0.$ These features of the family are needed in \cite{HelNdi1} to apply bifurcation theory to identify them as constrained Willmore minimizing tori.
\begin{figure}
\vspace{0.5cm}
\includegraphics[width= 0.3\textwidth]{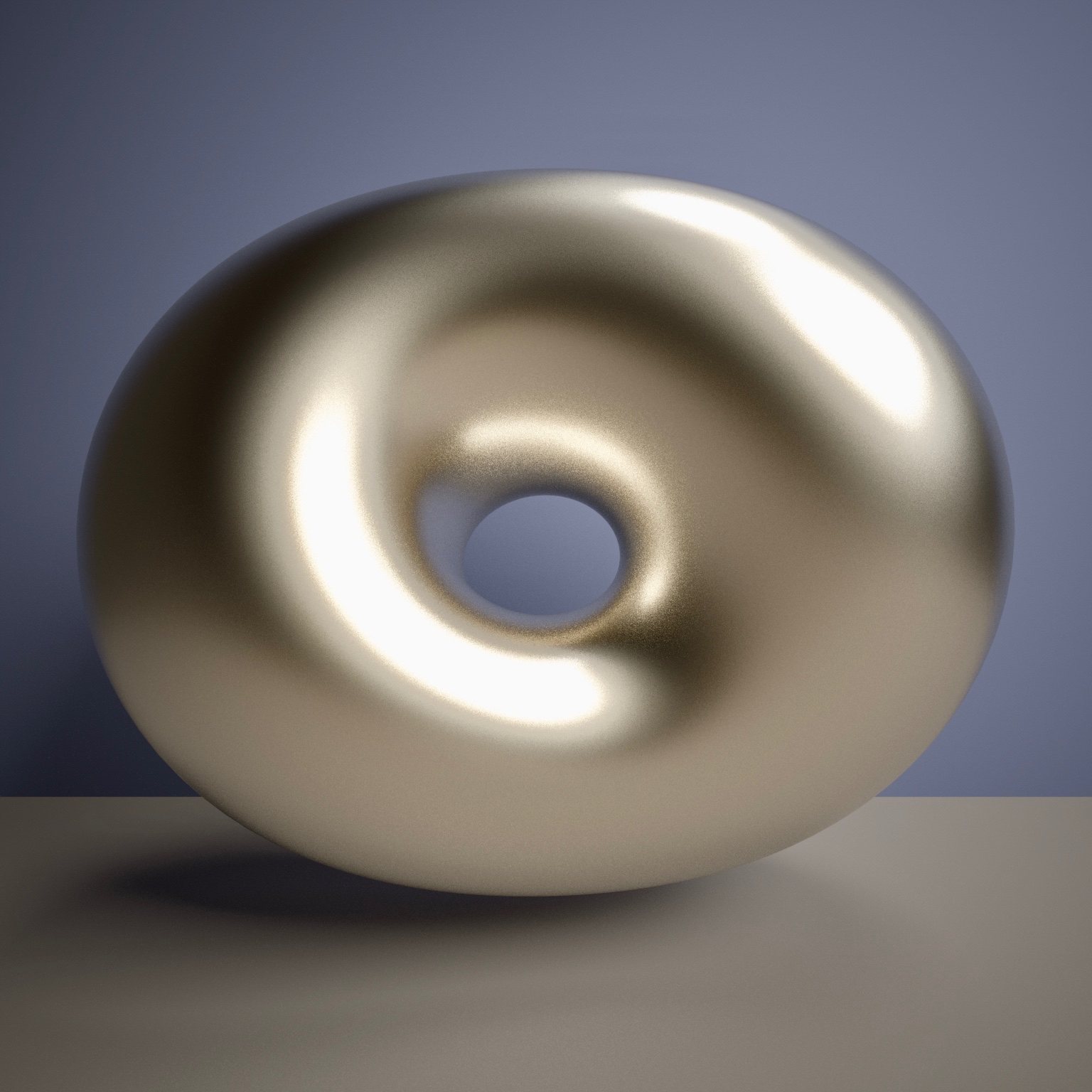}\hspace{1.5cm}
\includegraphics[width= 0.3\textwidth]{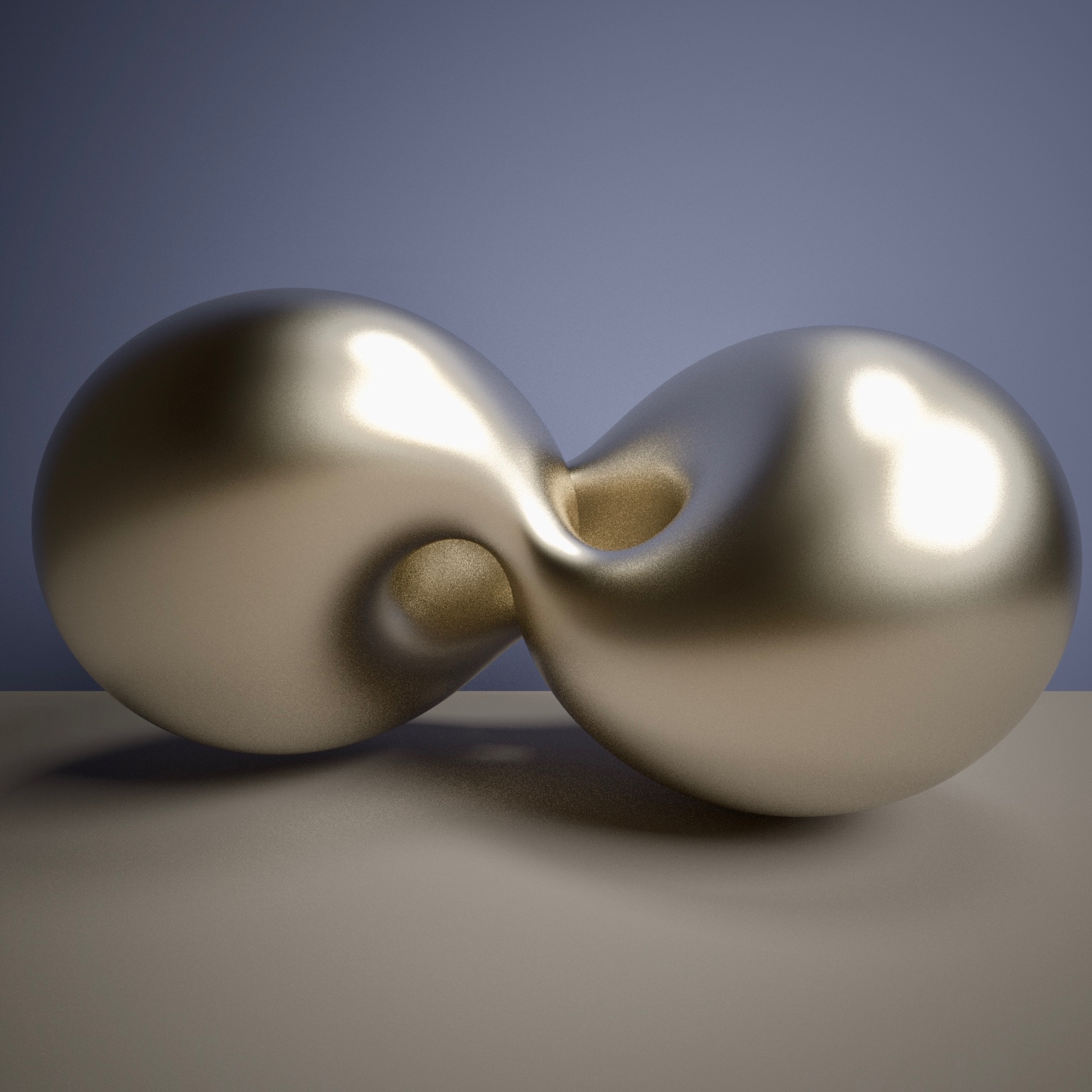}
\caption{ 
Two $(1,2)$-equivariant constrained Willmore tori (with intrinsic period $1$). The tori lie in a $2$-parameter family of surfaces deforming the Clifford torus. This family minimizes the Willmore functional in the respective conformal classes for surfaces ``close enough" to the Clifford torus.  Images by Nick Schmitt.}
\label{1lobe}
\end{figure}

The paper is organized as follows: in the first section basic integrable systems notations for constrained Willmore surfaces are introduced. We define equivariant immersions and state the Euler-Lagrange equations in terms of the conformal Hopf differential. The integrable structure of the equations is hereby encoded in the associated family of solutions. In Section 3 we construct constrained Willmore Hopf tori deforming homogenous tori $f^b$ and determine their normal variation at $f^b$. The forth section contains the proof of the main Theorem which is divided into four steps:  

\begin{enumerate}
\item construction of a $2$-parameter family $f_{(a,b)}$ of $(1,2)$-equivariant constrained Willmore tori associated to constrained Willmore Hopf cylinders.\\
\item show that candidates cover an open set close to the Clifford torus in the moduli space of conformal tori.
\item computation of the normal variation $\;\big(\del_{\sqrt{a}}f_{(a,b)} |_{a=0}\big)^\perp\;$ for the constructed family at the homogenous torus $f^b$ for fixed $b \sim1\; $ and $b\neq 1$.
\item show that the corresponding Lagrange multipliers $\;\alpha_{(a,b)}\;$ converge to $\;\alpha^b\;$ from below as $\;a \searrow 0$.\\
\end{enumerate}

We have included \cite[Section 3]{HelNdi1} dealing with the stability computations of the penalized Willmore functional at homogenous tori  as an appendix to make the paper more self-contained. These computations is used to identify the normal variations of the families of constrained Willmore Hopf tori constructed in Section 3.
\subsection*{Acknowledgments:} We want to thank Prof. Dr. Franz Pedit and Prof. Dr. Reiner Sch\"atzle for bringing our attention to the topic of this paper and for helpful discussions. We would also like to thank Dr. Nicholas Schmitt for supporting our theoretical work through computer experiments that helped finding the right properties for candidates surfaces to be minimizer and for producing the images used here. Moreover, we thank Dr. Sebastian Heller for helpful discussions.

\section{Preliminaries}\label{candidates}
We start with fixing some notations and basic properties of equivariant constrained Willmore surfaces.

\begin{Def}\label{S1action}
A map $\;f: \C \longrightarrow S^3 \;$  is called $\;\R$-equivariant, if there exist group homomorphisms 
\begin{equation*}
\begin{split}
&M: \R \longrightarrow \text{M\"ob}(S^3), \quad t \longmapsto M_t,\\
&\tilde M: \R \longrightarrow \{\text{conformal transformations of } \C\}, \quad t \longmapsto \tilde M_t,
\end{split}
\end{equation*}
such that 
$$f \circ \tilde M_t = M_t \circ f,  \text{ for all } t.$$
Here $\;$M\"ob$(S^3)\;$ is the group of M\"obius transformations of $\;S^3.$
\end{Def}

\begin{Rem}
If the map $\;f: \C \longrightarrow S^3\;$ is doubly periodic, then the resulting surface is a torus. A necessary condition for doubly periodicity of $\;f\;$ is that both $\;M_t\;$ and $\;\tilde M_t\;$ are periodic in $\;t,\;$ see \cite[Section 5]{Hethesis}. The possible periodic $1$-parameter subgroups $\;M_t\;$ and $\;\tilde M_t\;$ that can appear in the above definition can be easily classified, see for example \cite[Section 5]{Hethesis}. Thus up to the choice of a holomorphic coordinate $\;z = x+iy\;$ and isometries of $S^3$ we can assume without loss of generality that an equivariant torus

$$f: \C/\Gamma \longrightarrow S^3 \subset \C^2,$$

 for a lattice $\Gamma\subset \C,$  is given by 
 
\begin{equation}\label{equivariant}
f(x,y) = \begin{pmatrix}e^{imx} &0\\0 & e^{ikx} \end{pmatrix} f(0,y),
\end{equation}

for coprime integers $\;m\;$ and $\;k.\;$  In this case we call $\;f\;$ an $(m,k)$-equivariant surface and 
the curve given by $\;\gamma(y) = f(0,y)\;$ is called the profile curve which has to verify certain closing conditions for $\;f\;$ to be a torus. This notion of equivariant surfaces includes the well known examples of surfaces of revolution ($m=0,\; k=1$) and the Hopf cylinders ($m=1,\; k=1$), as for example discussed in \cite{He1}. 
\end{Rem}

\begin{Rem}
A $\R$-equivariant immersion $\;f: \C \longrightarrow S^3\;$ such that $\;M(\R)\;$ is not a periodic subgroup of M\"ob$(S^3)$ but is smoothly close to a $S^1$-equivariant surface is still of the form \eqref{equivariant} (up to conjugacy) with $\;m,\;k \in \R.\;$ This is due to the fact that  whether a $\;M \in SO(4,1)\;$ lie in the conjugacy class of $\;\tilde M \in SO(4)\;$ reduce to a condition on the trace of $\;M\;$ lying in a certain intervall.
\end{Rem}

In conformal geometry surfaces mapping into the conformal $S^3$ have two invariants which determine the surface up to M\"obius transformations, see  \cite[Theorem 3.1]{BuPP}.  The first one is the conformal Hopf differential $\;q.\;$ The second is the Schwarzian derivative $c$. In the equivariant case, the conformal Hopf differential $q$ determines the Schwarzian derivative $c$ up to a complex integration constant by the Gau\ss-Codazzi equations.
Thus we will only use $q$ in the following. In contrast to \cite{BuPP} we consider the conformal Hopf differential as a complex valued function by trivializing the canonical bundle $K_{\C/\Gamma}$ via $dz$.\\
\begin{Def}
Let $f:M \longrightarrow S^3$ a conformal immersion. The function
 $$q := \frac{\II \left(\frac{\del }{\del z},\frac{\del}{\del z} \right)}{|df|} $$
 is called the conformal Hopf differential of $f.$ \\
\end{Def}

\begin{Rem}\label{kurvenq}
For equivariant tori the conformal Hopf differential as well as the Schwarzian derivative depend only on the profile curve parameter and is periodic, see \cite[Proposition 3]{He1}.\\
\end{Rem}

\begin{Exa}\label{rotaionalHopf}
By definition the conformal Hopf differential of surfaces of revolution is real valued and given by $\;4q  = \kappa\;$ where $\;\kappa\;$ is the curvature of the (arclength parametrized) profile curve $\gamma$ in the upper half-plane viewed as the hyperbolic plane. In the same way one can compute for Hopf cylinders that $\;4q = \kappa  + \sqrt{G}i,\;$ where $\;\kappa\;$ is the geodesic curvature of the corresponding profile curve in a round $2$-sphere of curvature $\;G > 0$.
\end{Exa}
\subsection{Equivariant constrained Willmore tori and their associated family}$\ $\\
For equivariant constrained Willmore tori we give the Euler-Lagrange equation in terms of their Schwarzian derivative. This equation has an invariance which defines an associated family of constrained Willmore surfaces to a given solution. We start by recalling the Euler Lagrange equation of equivariant constrained Willmore surfaces  \cite[Equation (34) and Equation (33b)]{BuPP} specified to the equivariant case:\\

\begin{The}[\cite{BuPP}]
Let $\;f: T^2\cong \C/\Gamma \longrightarrow S^3\;$ be a conformally parametrized equivariant immersion and $\;q\;$ its conformal Hopf differential. Then $\;f\;$ is constrained Willmore if and only if there exists a $\;\mu + i \lambda\in \C\;$ such that $\;q\;$ satisfies the equation: 

\begin{equation}\label{EL}
\begin{split}
q'' +  8 \big(|q|^2 + &C\big)q - 8 \xi q = 2Re\big((-\mu+ i \lambda) q\big),\\
2\xi' &= \bar q'q - q' \bar q,
\end{split}
\end{equation}

\noindent
where $\;\xi\;$  is a purely imaginary function, $\;C\;$ a real constant and the derivative is taken with respect to the profile curve parameter.\\

\end{The}

\begin{Rem}The real part of Equation (\ref{EL}) is the actual constrained Willmore Euler-Lagrange equation. The imaginary part of the equation is the Codazzi equation and the equation on $\;\xi\;$ is the Gau\ss  $\ $equation. 
The Euler-Lagrange equation for general surfaces can be found in \cite{BuPP}.
For $(m,k)$-equivariant tori the function $\;\xi\;$ is given by 

$$\xi = i \tfrac{mk}{4}H,$$

 where $\;H\;$ is the mean curvature of the immersion into $\;S^3\;$ and $\;C = -\tfrac{1}{4}(m^2+k^2),\;$ as computed in \cite[Theorem 6.1 and Theorem 7.1]{Hethesis}.  \\
\end{Rem}


Let $\;f:  T^2 \longrightarrow S^3\;$ be an equivariant constrained Willmore surface with conformal Hopf differential $\;q.\;$  Consider $\;f\;$ as a doubly periodic immersion from $\;\C\;$ into $S^3$. By relaxing both periodicity conditions, i.e., by allowing general profile curves and real numbers for the equivariance type $\;m\;$ and $\;k,\;$ we obtain for $\;e^{i \theta} \in S^1\;$ a circle worth of associated constrained Willmore surfaces $\;f_\theta\;$ for a given $\;f,\;$ the so called constrained Willmore associated family, see \cite[Equation (35)]{BuPP}. These are obtained as follows:\\

\noindent
Let $\;q\;$ be a solution of \eqref{EL} and 
let $\;q_\theta\;$  be the family of complex functions given by

$$q_\theta =  q e^{2i\theta}, \quad e^{i \theta} \in S^1.$$  

Moreover, let

\begin{equation}\label{Lagrange}
\begin{split}
C_\theta &=  C + \tfrac{1}{8}\text{Re}((e^{4i\theta}-1) (-\mu-i \lambda))\\
\xi_\theta &=  \xi + \tfrac{1}{8}\text{Im}((e^{4i\theta}-1)(-\mu-i \lambda))\\
-\mu_\theta + i\lambda_\theta &= e^{-4i\theta}(-\mu + i \lambda).
\end{split}
\end{equation}

\noindent
Then $\;q_\theta\;$ satisfies Equation \eqref{EL}  with parameters $\;C_\theta, \;\mu_\theta, \;\lambda_\theta\;$ and function $\;\xi_\theta.\;$ In particular, the function $\;q_\theta\;$ and $\;\xi_\theta\;$ satisfies the Gau\ss-Codazzi equations for surfaces in $\;S^3.\;$ Thus there exist a family of surfaces $\;f_\theta\;$ with conformal Hopf differential $\;q_\theta\;$ and mean curvature given by $\;\xi_\theta.\;$ The so constructed surfaces $\;f_\theta\;$ are automatically constrained Willmore for every $\;\theta \in \R.\;$\\ 

\begin{Def}[Constrained Willmore Associated Family]\label{associatedfamily}$\ $\\
Let $\;f\;$ be a constrained Willmore surface and $\;q\;$ its conformal Hopf differential. The family of surfaces $\;f_\theta,\;$ $\theta \in \R\;$ determined by the conformal Hopf differential $\;q_{\theta} = q e^{2i\theta}\;$ is called the constrained Willmore associated family of $\;f.$
\end{Def}

\noindent
Surfaces with the same conformal Hopf differential and the same Schwarzian derivative (which is determined by the function $\;\xi\;$ and the real constant $\;C\;$) differ only by a M\"obius transformation, see \cite[Theorem 3.1]{BuPP}. As a consequence, since for an equivariant initial surface both invariants depend only on one parameter, 
all surfaces in the associated family of an equivariant constrained Willmore surface are also equivariant (and constrained Willmore). In general, these surfaces are not closed, i.e., $\;f_\theta: \C \longrightarrow S^3\;$ is not doubly periodic, even if the initial surface is.
Moreover, since a non-isothermic (or non-degenerate) surface $\;f\;$ is already determined up to M\"obius transformations by its conformal Hopf differential, the map $\;\theta \longmapsto f_\theta \in C^\infty(T^2, S^3)\;$ is  (in this case) smooth, see \cite[Theorem 3.3]{BuPP}.
\subsubsection{{\bf The associated family of homogenous tori $f^b$}}\label{Tr}$\ $\\The homogenous tori are given by the direct product of two circles with different radii. They can be parametrized by

$$f^b: \C/ (2\pi r \Z + i 2\pi s \Z), \quad f(x,y) = \begin{pmatrix}r e^{i\tfrac{x}{r}}, & s e^{i\tfrac{y}{s}}\end{pmatrix},$$

for $\;r,\;s \in \R\;$ satisfying $\;r^2+s^2=1\;$ with rectangular conformal class given by $\;b= \tfrac{s}{r}.\;$ The conformal Hopf differential (in this particular parametrization) is given by 

\begin{equation}\label{mutheta}
q=\tfrac{1}{2 rs} \quad \text{and} \quad \mu = \frac{s^2-r^2}{rs}.
\end{equation}

Since homogenous tori are isothermic, the Lagrange multiplier $\;\lambda\;$ in \eqref{EL} can be chosen arbitrarily and $\;q\;$ does not determine the surface uniquely. Thus the associated family $\;f_\theta^b\;$ is also not uniquely determined and is only smooth in $\;\theta\;$ for appropriate $\;\xi\;$ and $\;C\;$ depending on both $\;\mu\;$ and $\;\lambda.\;$ In fact, the following choice of associated family of $\;f^b\;$ seen as a map from $\;\C\;$ to $\;S^3\;$ is smooth in $\;\theta$: 

\begin{equation}\label{associated1}f_\theta^b(x,y) = \begin{pmatrix}r^\lambda_\theta e^{i\left(\tfrac{1}{r^\lambda_\theta} \cos(\theta)x-\tfrac{1}{r^\lambda_\theta}\sin(\theta)y\right)}, & s^\lambda_\theta e^{i\left(\tfrac{1}{s^\lambda_\theta} \sin(\theta)x+\tfrac{1}{s^\lambda_\theta}\cos(\theta)y\right)}\end{pmatrix}.
\end{equation}

Here $\;r^\lambda_\theta,\; s^\lambda_\theta\;$ are determined by the Lagrange multiplier $\;\tilde \mu_\theta\;$  via equation
\eqref{mutheta}, where $\;\tilde \mu_\theta\;$, $\tilde \lambda_\theta\;$ satisfies 

$$-\tilde \mu_\theta + i \tilde \lambda_\theta = e^{-2i \theta}(-\mu + i \lambda).$$ 

Our choice of the associated family ensures that for a smooth family $\;f^t\;$ of non-degenerate surfaces converging to \;$f^b \;$\; which is degenerate as $t\longrightarrow 0$\; , i.e., the map 
$$\;(0, t_0)\longrightarrow C^\infty_{Imm}, \; t \longmapsto f^t\;$$ is smooth and  $t\longrightarrow 0=f^b$ in \; $C^\infty_{Imm}$, 
also the associated family $\;f^t_\theta\;$ has the same regularity in $\;t\;$ for every fixed $\;\theta\;$ and $\;{\displaystyle \lim_{t \rightarrow 0 }}f_\theta^t = f^b_\theta.\;$ In particular, the corresponding Lagrange multipliers $\;\mu^t_\theta\;$ and $\;\lambda^t_\theta\;$ are continuous for $\;t \in [0, t_0)\;$ for every fixed $\;\theta.\;$ The equivariance type of the ``rotated'' surface $\; f^b_\theta \;$ is given by

\begin{equation}\label{cliffequiv}
\frac{m}{k} = \begin{vmatrix}\frac{\cos(\theta)s^\lambda_\theta}{\sin(\theta) r^\lambda_{\theta}}\end{vmatrix} \in [0, 1].\end{equation}

Note that the derivative of the equivariance type by $\;\theta\;$  at $\;\theta_0 \in (0, \pi/4)\;$ vanishes  if and only if $\;\lambda_{\theta}^0 =0$.

\section{Constrained Willmore Hopf cylinders.}\label{Hopf}Since tori of revolution are isothermic, they cannot cover an open set of the Teichm\"uller space. In contrast, all conformal types can be realized as (constrained Willmore) Hopf tori (i.e., $\;m=n=1$), see \cite{Pinkall, He2}. 
The Willmore energy of the Hopf cylinder reduces to the (generalized) energy  of the curve in $\;S^2\;$ and the conformal type of the torus translates into invariants of the curve, namely length and oriented enclosed area.  Thus a Hopf torus is constrained Willmore if and only if there exist Lagrange multipliers $\;\lambda\;$ and $\;\mu \in \R\;$  such that the geodesic curvature $\;\kappa\;$ of its profile curve $\;\gamma\;$ in $\;S^2\;$ (of constant curvature $\;G$) satisfies 

\begin{equation}\label{elastic}
\kappa'' + \tfrac{1}{2}\kappa^3 + (\mu + \tfrac{G}{2})\kappa + \lambda = 0.
\end{equation}

This equation can also be deduced from Equation \eqref{EL} for $\;4q= \kappa + i \sqrt{G}.\;$ The Lagrange multiplier $\;\mu\;$ is the length constraint and $\;\lambda\;$ is the enclosed area constraint. We call curves (not necessarily closed) into the round $\;S^2\;$ (with constant curvature $\;G$) satisfying equation \eqref{elastic} constrained elastic. \\

Since we are interested in periodic solutions of \eqref{elastic}, we can restrict ourselves to the initial values 

\begin{equation}\label{IVP}
\kappa'(0) = 0 \quad \text{and} \quad \kappa(0) = \kappa_0
\end{equation}
for the Euler-Lagrange equation. 

\begin{Rem}\label{parameterdependence}
The unique solution $\;\kappa(x, \kappa_0, \lambda, \mu)\;$ of the initial value problem given by \eqref{elastic} and \eqref{IVP} depends in a real analytic way on the parameters $\;\lambda,\;  \mu\;$ and the initial value $\;\kappa_0$. 
\end{Rem}

Let  $\;\gamma: \R \longrightarrow S^2\;$ be a curve and $\;\kappa\;$ its geodesic curvature. We use an integrated version of the Euler-Lagrange equation for constrained elastic curves obtained by multiplying \eqref{elastic} with $\;\kappa'\;$ and integrate. The curve $\;\gamma\;$ is therefore constrained elastic if and only if there exist real numbers $\;\mu\;$, $\lambda\;$ and $\;\nu\;$ such that $\;\kappa\;$ satisfies

\begin{equation}\label{elastic2}
(\kappa')^2 + \tfrac{1}{4}\kappa^4 + (\mu + \tfrac{G}{2})\kappa^2 + 2\lambda \kappa + \nu = 0.
\end{equation}

The freedom of the integration constant $\;\nu\;$ corresponds to the initial value $\;\kappa_0,\;$ which must be a root of the polynomial $\;P_4 = \tfrac{1}{4}\kappa^4 + (\mu + \tfrac{G}{2})\kappa + 2\lambda \kappa + \nu$. \\

\begin{Rem}
The conformal Hopf differential of the Clifford torus in $(1,1)$-parametrization $\;q_{(1,1)}\;$  is the conformal Hopf differential of the Clifford torus considered as a torus of revolution $\;q_{(1,0)}\;$ multiplied by the imaginary unit $i,$ see Example \ref{rotaionalHopf}. Thus the role of the Lagrange multipliers $\;\lambda\;$ and $\;\mu \;$ switch compared to Section \ref{Tr}. 
\end{Rem}

All constrained elastic curves in $\;S^2\;$ can be parametrized in terms of the Weierstrass elliptic functions and limits of these. 
Elliptic functions are defined on a torus $\;\C/\Gamma,\;$ where the lattice $\;\Gamma\;$ is determined by its (in general complex) lattice invariants $\;g_2\;$ and $\;g_3.\;$ For constrained elastic curves these invariants are computed in \cite[Lemma 1]{He2} to be

\begin{align}\label{invariants}
g_2 &= \frac{(\mu + \tfrac{G}{2})^2}{12} + \frac{\nu}{4} \\
g_3 &=\frac{1}{216}(\mu+ \tfrac{G}{2})^3 + \frac{1}{16}\lambda^2 - \frac{1}{24}\nu (\mu + \tfrac{G}{2}).
\end{align}

The lattice $\;\Gamma\;$ is non degenerated, i.e., has two real linear independent generators,  if and only if its discriminant $\;D := g_2^3 - 27g_3^2\neq 0.\;$ In this case we denote the generators of the lattice by $\;2 \omega_1, 2 \omega_2 \in \C.\;$ Since by construction $\;g_2, g_3 \in \R,\;$ the resulting lattice $\;\Gamma\;$ is either rectangular or rhombic. Thus we can fix $\;2\omega_1\in \R\;$ and there exits a smallest lattice point on the imaginary axis, which we denote by  $\;2\omega_3 \in i\R.\;$ For details on elliptic functions see \cite{KoecherKrieg}. Now we can parametrize all solutions of \eqref{elastic2} with periodic curvature $\;\kappa\;$ such that $\;D = g_2^3 - 27 g_3^2 \neq 0\;$ according to the following theorem \cite[Theorem 2]{He2}. \\

\begin{The}[\cite{He2}]\label{curves}
For $\;g_2,\;$ $g_3\;$ and $\;\kappa_0 \in \R\;$ with $\;D = g_2^3 - 27 g_3^2 \neq 0,\;$ the curve $$\gamma = [\gamma_1: \gamma_2]: \R \longrightarrow  \CP^1$$ with $\;\gamma_i: \R \longrightarrow \C\;$ given by

\begin{equation}
\begin{split}
\gamma_1 &=  \frac{\sigma(x+x_0 - \rho)}{\sigma(x+x_0)}e^{\zeta(\rho)(x+x_0)}\\
\gamma_2 &= \frac{\sigma(x+x_0 + \rho)}{\sigma(x+x_0)}e^{\zeta(-\rho)(x+x_0)}
\end{split}
\end{equation}

is constrained elastic in the round $\;S^2 \cong \C P^1\;$ with curvature $\;G (> 0).\;$
Hereby $\;\sigma\;$ and $\;\zeta\;$ denote the Weierstrass $\;\sigma$- and $\;\zeta$-function respectively. The parameters $\;x_0, \rho \in i\big(0, |\omega_3|\big) \;$ satisfy 

$$2\wp(x_0) + \wp(\rho) + \tfrac{1}{4}\kappa_0^2 = -\tfrac{1}{4}G < 0,$$

and $\;\omega_3\;$ is the lattice point of $\;\Gamma\;$ with smallest length lying on the imaginary axis and the upper half plane.  Moreover, all constrained elastic curves in $\;S^2\;$ with $\;D = g_2^3 - 27 g_3^2 \neq 0\;$ and periodic curvature can be obtained this way.\\
\end{The}

\begin{Rem}
For given lattice invariants $\;g_2\;$ and $\;g_3\;$ we obtain thus a $2$-parameter ($x_0\;$ and $\;\rho$) family of (not necessarily closed) constrained elastic curves into a $2$-sphere of constant curvature $\;G  (>0).\;$ It is shown in \cite[Corollary 5]{He2} that there exists an unique $\;\hat x_0 \in i\big(0, |\omega_3| \big)\;$ such that the corresponding curve becomes elastic, i.e., $\;\lambda= 0.\;$  In this case we obtain 
$$\wp(\rho) = \wp(\omega_3) - \tfrac{1}{4}G.$$
 Moreover, by \cite{He2} we have  
 $$\wp(\rho) = \tfrac{1}{6} (\mu - G)$$
 
 for all $\;\lambda.\;$ If not otherwise stated we will fix $\;G=1\;$ in the following. \end{Rem}

A straightforward computation, see \cite[Proposition 4]{He2}, shows that the curve $\;\gamma,\;$  given by Theorem \ref{curves}, closes if and only if there exist $\;m, n \in \N\;$ such that

\begin{equation}\label{monodromy}
M(g_2, g_3, \rho) := \rho \eta_1 - \zeta(\rho)\omega_1 = i \tfrac{m}{2n}\pi,
\end{equation}

where $\;\zeta\;$ is again the Weierstrass $\;\zeta$-function and $\;\eta_1 = \zeta(\omega_1).\;$ $M\;$ is called the monodromy of the curve. Geometrically speaking, $\;m\;$ is the winding number of the curve and $\;n\;$ is the lobe number, i.e., the number  of (intrinsic) periods of the curvature till the curve closes in space.\\ 

For the Clifford torus the profile curve is (a piece of a) geodesic in $\;S^2\;$ and can be described using trigonometric functions. In this case we have $D= 0$ and the curve is given as the limit curve as $\;\omega_3,\;$ the smallest lattice point lying on the imaginary axis, goes to infinity. In other words, the imaginary part of $\;\tau := \tfrac{\omega_1 + \omega_3 }{2\omega_1}\;$ goes to infinity (while the real part is constantly $\;\tfrac{1}{2}$).
In this case the corresponding limits of the Weierstra\ss $ $ elliptic functions and invariants are given by, see \cite[Chapter 13.15.]{Zentgraf}:

\begin{align}
\wp_{\infty}(z) &= -a  + 3a \frac{1}{\sin^2(\sqrt{3a}z)}\\
\zeta_{\infty}(z) &= az + \sqrt{3a}\frac{\cos(\sqrt{3a}z)}{\sin(\sqrt{3a}z)}\\
\omega_1 &= \tfrac{1}{\sqrt{12a} }\pi \\
\eta_1&= \tfrac{a}{\sqrt{12a}}\pi
\end{align}

for a real number $\;a\;$ with 

\begin{equation}\label{g2g3}
g_2^\infty = 12 a^2 \quad \text{and}\quad  g_3^\infty = 8 a^3.
\end{equation} 

Since for the Clifford torus we have $\;\nu_\infty = 0\;$ (and thus $\;144a^2=12 g_2^\infty = (\mu_\infty + \tfrac{1}{2})^2$) we obtain by \eqref{invariants} that 

\begin{equation}\label{rhoinfty1}
\wp_\infty(\rho_\infty)= \tfrac{1}{6}(\mu_\infty-1) = 2a- \tfrac{1}{4},
\end{equation}
from which we can compute

\begin{equation}\label{rhoinfty2}
\rho_\infty = \frac{1}{\sqrt{3a}} \text{arsin}\left(\sqrt{\tfrac{12 a}{12a-1}}\right).
\end{equation}

 \begin{Rem}
The $\;\rho \in i \big(0, |\omega_3|\big)\;$ we use in Theorem \ref{curves} is unbounded for $\;D \longrightarrow 0.\;$ But there is another representative 

$$\tilde \rho  = \rho - (\omega_3-\omega_1) =  \rho \mod \quad \Gamma$$

 that is actually bounded and the corresponding monodromy equations \eqref{monodromy} are equivalent. Since the meaning of the parameter remains the same, we will still denote $\;\tilde \rho\;$ by $\;\rho\;$ in the following.
 \end{Rem}

By Equation \eqref{rhoinfty1} and because $$\wp(\rho) = \wp(\omega_3) -\tfrac{1}{4} $$ for elastic curves 

$$\lim_{\tau \rightarrow \infty} \wp (\omega_3) \longrightarrow 2a.$$

 Thus the invariants for any family of constrained Willmore Hopf tori converging (smoothly) to the Clifford torus satisfy $\;\omega_3 = \omega_1\;$ mod $\;\Gamma,\;$ i.e., these are wavelike solutions with $\;D \nearrow 0,\;$ see \cite[Chapter 13.15.]{Zentgraf}. The closing condition \eqref{monodromy} converges to 

\begin{equation}\label{rhoinfty3}
\rho_\infty  \tfrac{a}{\sqrt{12a}}\pi - \left (a\rho_\infty + \sqrt{3a}\frac{\cos(\sqrt{3a}\rho_\infty)}{\sin(\sqrt{3a}\rho_\infty)} \right) \tfrac{1}{\sqrt{12a} }\pi  = \tfrac{m}{2n} \pi i.
\end{equation}

For the simply wrapped Clifford torus, i.e., $\;m= 1,\;$ the Equations \eqref{rhoinfty1}, \eqref{rhoinfty2} and \eqref{rhoinfty3} yields 

$$a= \tfrac{n^2}{12}\text{ (and $n >1$)  or equivalently }\mu^n_\infty = n^2 - \tfrac{1}{2}.$$

Since the ratio of winding number and lobe number is rational for closed solutions, it remains constant throughout the deformation induced by a continuous deformation of the parameters $\;g_2,\; g_3\;$ and $\;\rho.\;$ Suppose there exists a family $\;f_t\;$ of embedded constrained Willmore Hopf tori given by Theorem \ref{elastic} converging smoothly to the Clifford torus as $\;t \longrightarrow 0,\;$ then  its Lagrange multiplier $\;\mu_t\;$ necessarily converges to $\;n^2 - \tfrac{1}{2},\;$ for an integer $\;n,\;$ and $\;\lambda_t, \nu_t \longrightarrow 0.\;$ The following theorem shows that this condition is also sufficient.\\

\begin{Rem}
For $D \neq 0$ the spectral curve of the corresponding constrained Willmore surface is given by (a covering) of the torus $\C/ (2\omega_1\Z \oplus 2\omega_2\Z)$ on which $\wp$  is defined, see \cite[Section 3.5 and Remark 11]{He2}. For $D\longrightarrow 0$ this torus degenerates to a sphere, corresponding to the spectral curve of a homogenous torus. To deform a homogenous torus in direction of constrained Willmore Hopf tori with $D\neq 0$ we have to open the double point on the sphere corresponding to $\mu= n^2-\tfrac{1}{2}.$
\end{Rem}
\begin{The}\label{smoothfamily2d}
Let the Clifford torus be parametrized by the formulas in Theorem \ref{curves} obtained by letting $\;(g_2, g_3, \kappa_0, \rho) \longrightarrow (g_2^\infty, g_3^\infty, 0, \rho^\infty)\;$ given by \eqref{g2g3} and \eqref{rhoinfty2}. Then for every integer $\;n > 1\;$ there is a $2$-parameter family of $n$-lobed and embedded constrained Willmore Hopf tori $\;f^n_{g_2, g_3}\;$ continuously deforming the Clifford torus.  Moreover, the limit Lagrange multiplier of the family at the Clifford torus is $\;\mu^n_\infty = n^2-\tfrac{1}{2}\;$ (and $\;\lambda^n_\infty = 0$).\\
\end{The}

\begin{proof}
For given real $\;g_2,\;$ $g_3\;$ with $\;D< 0\;$ we can define constrained elastic curves and their corresponding monodromy. The Weierstrass elliptic functions converge for $\;D\nearrow 0\;$ (and real $\;g_2,\;$ $g_3$) uniformly on every compact set, see \cite{KoecherKrieg}.
Thus the monodromy function $\;M(g_2, g_3, \rho), \;$ see \eqref{monodromy}, remains differentiable in $\;\rho\;$ and is continuous in $\;g_2\;$ and $\;g_3\;$ as $\;
D \nearrow 0.$\\

We are interested in the singular case (i.e., $\;D=0$) of the Clifford torus. In this case it is computed above that 

$$g_2^\infty =  \tfrac{n^4}{12} \quad \text{ and }\quad g^\infty_3 = \tfrac{n^6}{216}.$$

We denote by $\;\rho^n_\infty\;$ the unique solution of 

$$\wp_\infty(\rho^n_\infty) = \tfrac{n^2}{6} - \tfrac{1}{4}\; \in \; i\big(0, |\omega_3|\big).$$

From this we obtain

\begin{equation}\label{dotmonodromy}
 \tfrac{\partial }{\partial \rho} M \big(\tfrac{n^4}{12, },\tfrac{n^6}{216}, \rho\big) \mathlarger{|}_{\rho = \rho^n_\infty} = \eta^\infty_1 + \wp_\infty(\rho^n_\infty) \omega_1 =  (\tfrac{n}{4} - \tfrac{1}{4n})\pi 
 \end{equation}
 
which is non zero for $\;n \neq1.$\\

Then we obtain by the implicit function theorem that for every $\;(g_2, g_3) \sim\big(\tfrac{n^4}{12 },\tfrac{n^6}{216}\big)\;$ with $\;D\leq0\;$ an uniquely determined $\;\rho(g_2, g_3)\;$ (continuously depending on $\;g_2
\;$ and $\;g_3$) such that $\;M\big(g_2, g_3, \rho(g_2, g_3)\big) =const,\;$ and $\rho\big(\tfrac{n^4}{12 },\tfrac{n^6}{216}\big)= \rho^n_\infty$ Therefore, we obtain a $2$-parameter family $\;f^n_{g_2, g_3}\;$ of curves depending continuously on $\;g_2\;$ and $\;g_3\;$  given by Theorem \ref{curves} satisfying the closing condition. \\
\end{proof}

\begin{Rem}
Using the same arguments as in the theorem it is possible to show the existence of  a $2$-parameter family of constrained Willmore Hopf cylinders bifurcating at homogenous tori for every $\;a \neq \tfrac{1}{12}$.  For $\;G=1\;$ the value of $\;a = \tfrac{1}{12}\;$ corresponds to M\"obius variations, i.e., $\;n=1$. \\\end{Rem}

\begin{Rem}\label{regularityM}
For $\;D\neq0\;$ the dependence of the curves on the parameters $\;g_2,$ $g_3\;$ is real analytic and thus $\;\rho\big(g_2,g_3\big)\;$ is then also real analytic. In the limit (for $\;D \longrightarrow 0$) the family $\;f^n_{g_2, g_3}\;$ converge smoothly to homogenous tori but at $\;D=0\;$ the regularity is more subtle. The parameters $\;\big(g_2, g_3, \rho\big)\;$ is real analytic in $\;\mu,\; \lambda, \;$ and $\;\kappa_0$. Recall that $\;\kappa\big(x, \mu, \lambda, \kappa_0 \big)\;$ depends real analytically on $\big(\mu, \lambda, \kappa_0\big)$. Thus also the monodromy $\;M\;$ is real analytic in the parameters $\;\big(\mu, \lambda, \kappa_0\big).$ Further, since $\;\wp(\rho) = \tfrac{1}{6}(\mu-1),\;$ we have that $\;\wp\;$ is for $\;\rho \in i\big(0,  |\omega_3|\big)\;$ (and for fixed $\;g_2\;$ and $\;g_3$) a local diffeomorphism from the $\;\rho\;$-line to the $\;\mu\;$-line. Thus in a first step we can change the parameters to $\Big(g_2, g_3, \rho(g_2, g_3)\Big) \longrightarrow \Big(g_2, g_3, \mu(g_2,g_3)\Big)$.\\

For $\;D\leq0\;$ the ``map" $\;\Psi\;$ that associate to

 $$\big(g_2, g_3, \mu\big) \longmapsto \big(\mu, \lambda, \kappa_0\big)$$
 
  is multi-valued in $\;\lambda\;$ and $\;\kappa_0\;$ and behaves rather like a square/fourth order root, see \eqref{invariants}. But $\Psi$ becomes a local diffeomorphism for $\;D \neq0\;$ and $\;\lambda \neq0\;$ by choosing an appropriate branch of the parameters $\lambda$ and $\kappa_0.$ We first observe that 
  
\begin{equation}\label{lambda-lambda}
\kappa\big(\mu, \lambda, \kappa_0\big) = - \kappa\big(\mu, -\lambda, -\kappa_0\big).
\end{equation}

 Therefore, we can choose without loss of generality to $\;\lambda \in \R_-\;$ and thereby fixing an orientation of the curves, where constant solutions have positive curvature. Moreover, recall that a necessary condition for the smooth convergence of the family $f^n_{g_2, g_3}$ to the Clifford torus is that the discriminant $D \longrightarrow 0$ from below, i.e., the family consists of wavelike solutions. In this case the $4$-th order polynomial  
 
$$P_4 = \tfrac{1}{4}\kappa^4 + \big(\mu + \tfrac{G}{2}\big) \kappa^2 + 2 \lambda \kappa + \nu$$

has only two real roots for $D< 0$ (or $1$ root with multiplicity $2$ for $D=0$), both leading to the same solution of the differential equation (up to translation of the curve parameter), see also \cite{He2}. Thus we can choose without loss of generality $\kappa_0$ to be the bigger root of $P_4$. \\

The vector field $\;\tfrac{\del }{\del \mu}\;$ on the parameter space $\;(g_2, g_3, \mu)\;$ is mapped via $\;\Psi\;$ to a vector field $\;\frac{\del }{\del \tilde \mu}\;$ in $\;(\mu, \lambda, \kappa_0)\;$. Because of Equations \eqref{invariants} defining $\;g_2\;$ and $\;g_3$, there exist unique (and smooth) functions $\;\lambda(\mu)\;$ and $\;\kappa_0(\mu)\;$ such that 

$$g_2\big(\mu, \lambda(\mu), \kappa_0(\mu)\big) = const \quad \text{and} \quad g_3\big(\mu, \lambda(\mu), \kappa_0(\mu)\big) = const.$$

Therefore we obtain
$$ \frac{\del M}{ \del \mu}\Big(\mu, \lambda(\mu),\kappa_0(\mu)\Big) = \frac{\del M}{\del \tilde \mu}.$$

Thus by changing the parameters

$$(\mu, \lambda, \kappa_0) \longrightarrow (\tilde \mu, \tilde \lambda, \tilde \kappa_0) = \big(\mu, \lambda + \lambda(\mu), \kappa_0 + \kappa_0(\mu) \big)$$

we obtain that there exists a smooth function $\tilde \mu(\tilde \lambda, \tilde \kappa_0)$ such that 

$$M\big(\tilde \mu(\tilde \lambda, \tilde\kappa_0),\tilde \lambda, \tilde \kappa_0\big) = const.$$

This new smooth family parametrized by $(\tilde \lambda, \tilde \kappa_0)$ multiply covers the old family $f^n_{g_2, g_3}$. 
Note that by \cite[Corollary 3]{He2} there are no solutions with $D>0$ with parameters $( \lambda, \kappa_0, \mu)$ close to  $\big( 0, 0, n^2 - \tfrac{1}{2}\big)$.\\
\end{Rem}

\begin{Def}\label{f^bg2g3}
In the following we rename the parameters $(\tilde \mu, \tilde \lambda, \tilde \kappa_0)$  by $(\mu, \lambda, \kappa_0)$ for simplicity. Thus we will denote the above introduced real analytic \;$2$-parameter family of constrained Willmore tori obtained by $f_{ \lambda, \kappa_0}^n.$  \\\end{Def}
\begin{Rem}
The surfaces $\;f_{\lambda, \kappa_0}^n\;$ is homogenous if and only if $\;D = g_2^3- 27g_3^2 =0.$ These surfaces can be identified within  $\;f^{n}_{\lambda, \kappa_0}$ by varying $\;a\sim n^2 - \tfrac{G}{2}\;$ (and still  prescribe $\;M\big(12a^2, 8a^3, \rho^n_{\infty}(a)\big) = \tfrac{1}{n} \pi i$). Thus for a given $\;\kappa_0 (>0)\;$ there exists a unique $\;\lambda = \lambda_0 \;$ such that the resulting surface  $f^{n}_{\lambda_0, \kappa_0}$ is homogenous (and $\;\kappa_0\;$ is the constant curvature of its profile curve).\\
\end{Rem}
\begin{Rem}
For $\;(g_2, g_3)\sim \big(g^\infty_2, g^\infty_3\big) \in \R^2\;$  such that $\;\big(g_2^\infty\big)^3 -27\big(g_3^\infty\big)^2 =0,\;$ and discriminant $\;D\big(g_2, g_3\big) = \big(g_2)^3 -27\big(g_3\big)^2<0\;$ let $\;c\big(g_2, g_3\big) = c\big(D(g_2, g_3)\big)\;$ be an arbitrary strictly positive and smooth function depending only on the discriminant.  Then by the same computations as above and the implicit function theorem there exist a unique function $\;\rho^{c(D)}(g_2, g_3)\;$  satisfying 

$$M\Big(g_2, g_3, \rho^{c(D)}(g_2, g_3)\Big) = c\Big(D(g_2,g_3)\Big) \pi i.$$

for all $\;(g_2, g_3) \sim \big(g^\infty_2, g^\infty_3\big)\;$ close enough with $\;D(g_2,g_3)< 0.$
Therefore we obtain a $2$-parameter family $\;f^{c(D)}_{g_2,g_3}\;$ of constrained Willmore Hopf cylinders. The family of surfaces is continuous in $\;g_2\;$ and $\;g_3\;$ and real analytic in the corresponding parameters  $\;\lambda\;$ and $\;\kappa_0\;$. The limit Lagrange multiplier $\;\mu_{\infty}\;$ can be computed from $\;g_2^\infty,\;$ $g_3^\infty\;$ and the limit monodromy $\;c(0)$. 
\end{Rem}

The conformal type of a Hopf torus is given by the lattice generated by the vector $\;2\pi \in \C\;$  and the vector $\;\tfrac{1}{2}(A + i L) \in \C,$ where $\;L\;$ is the length of the corresponding profile curve in $\;S^2\;$ and $$\;A= \int_I \kappa(s) ds\;$$ the oriented enclosed volume, see \cite[Proposition 1]{Pinkall}. These quantities can be explicitly computed, see \cite[Theorem 5]{He2}.\\

\begin{Pro}\label{openneigh}
For every lobe number $n \geq 2$ the map $\tilde \bigPi$ which assigns to $(\lambda, \kappa_0)$ the conformal class of the immersions $f^n_{\lambda, \kappa_0},$  see Definition \ref{f^bg2g3}, 
covers an open neighborhood of the rectangular conformal class \;$(0,b)$\; for all \;$b \in \R_{>1}$. 
\end{Pro}
\begin{proof}

By construction we have 

$$\tilde \Pi(\lambda, \kappa_0) = \tfrac{1}{2}A\big(\lambda, \kappa_0\big) + i \tfrac{1}{2}L\big(\lambda, \kappa_0\big),$$ 

where $\;L\;$ is the length and $\;A\;$ the oriented enclosed area of the profile curve of \;$\;f^n_{\lambda, \kappa_0}\;$ in the $2$-sphere with constant curvature $\;G = 1.$ 

We first slightly reparametrize the family $\;f^n_{\lambda, \;\kappa_0}$.  For given $\;\kappa_0\;$ there exists a unique $\;\lambda(\kappa_0)\;$ such that the resulting surface $\;f^n_{\lambda(\kappa_0), \;\kappa_0}\;$ is homogenous. Let \;$\;\tilde f^n_{ \tilde \lambda,\kappa_0}:= f^n_{\tilde \lambda + \lambda(\kappa_0),\;\kappa_0}.\;$ Then  we obtain 

$$\partial_{\kappa_0} A\Big(\tilde \lambda+\lambda(\kappa_0),\;  \kappa_0\Big){{|}}_{( \tilde  \lambda,\;\kappa_0) = (0, \kappa_0)} \neq 0 \quad \text{for} \quad \kappa_0 \neq0.$$ 
This implies that for \;$ \kappa_0\neq 0$, there exists a real analytic map $$ \;\overline\kappa_0: (-\varepsilon_{\kappa_0}, \varepsilon_{\kappa_0})\longrightarrow (\kappa_0-\varepsilon_{\kappa_0}^1, \kappa_0+\varepsilon_{\kappa_0}^1)\;$$ such that  

$$\;A\Big(\tilde \lambda+\lambda(\overline\kappa_0(\tilde \lambda)),\;  \overline\kappa_0(\tilde \lambda)\Big) = A\Big(\lambda(\kappa_0),\;  \kappa_0\Big), \;\forall \;\tilde\lambda \in (-\varepsilon_{k_0}, \varepsilon_{k_0}),$$ with \;$\varepsilon_{\kappa_0}$ and \;$\varepsilon_{\kappa_0}^1$\; positive numbers and \;$\kappa(0)=\kappa_0$. 
Since the function 

$$ (-\varepsilon_{\kappa_0}, \varepsilon_{\kappa_0}) \ni \tilde \lambda \longmapsto  L\Big(\tilde \lambda+\lambda(\overline\kappa_0(\tilde\lambda )), \;\overline\kappa_0(\tilde \lambda)\Big)$$ is real analytic, we have that either there exists \;$m\geq 1$\; such that 
\begin{equation}\label{nonconstant}
\frac{d^mL\Big(\tilde \lambda+\lambda(\overline\kappa_0(\tilde \lambda)),\;  \overline\kappa_0(\tilde \lambda)\Big)}{d\tilde \lambda^m}{{|}}_{\tilde  \lambda= 0} \neq 0, \;\;\text{and}\;\;\frac{d^i L\Big(\tilde \lambda+\lambda(\overline\kappa_0(\tilde \lambda)),\;  \overline\kappa_0(\tilde \lambda)\Big)}{d\tilde \lambda^i}{{|}}_{\tilde  \lambda= 0} = 0\;\;\text{for}\;\;1\leq i\leq m-1.
\end{equation}
or
\begin{equation}\label{constant}
\;L\Big(\tilde \lambda+\lambda(\overline\kappa_0(\tilde \lambda)),\;  \overline\kappa_0(\tilde \lambda)\Big)=\;L\Big(\lambda(\kappa_0),  \kappa_0\Big)\;\;\; \forall\; \tilde \lambda \in (-\varepsilon_{\kappa_0}, \varepsilon_{\kappa_0}).
\end{equation}

Consider the normal variation $\;\hat \varphi\;$ of the family $\;\tilde f^n_{\tilde \lambda,  \overline\kappa_0(\tilde \lambda)}\;$ with respect to $\;\tilde \lambda\;$ at a homogenous torus. 
Since the profile curve of a homogenous torus in $S^2$ is a circle and the enclosed area of $\;\tilde f^n_{\tilde \lambda,  \overline \kappa_0(\tilde \lambda)}\;$ is constant, we have by the isoperimetric problem  that $\hat\varphi$ satisfies $\;\delta L\big(f^b\big) (\hat \varphi) = 0\;$ and $\;\delta^2 L\big(f^b\big)(\hat \varphi, \hat \varphi) >0$ which implies \eqref{nonconstant}.

This  gives for $\kappa_0 \neq0$ the existence of a right-neighborhood  \; $[0, \varepsilon_{\kappa_0}^2)$\;  of \;$0$\; such that  the image of \;$[0, \varepsilon_{\kappa_0}^2)$\; under the map 

$$\tilde \lambda \in (-\varepsilon_{\kappa_0}, \varepsilon_{\kappa_0})\longrightarrow  L\Big(\tilde \lambda+\lambda(\overline\kappa_0(\tilde\lambda )), \overline\kappa_0(\tilde \lambda)\Big)$$
is a right -neighborhood of \;$L_{\kappa_0}:=L\Big(\lambda(\kappa_0),  \kappa_0\Big)$\; denoted by\; $[L_{\kappa_0}, L_{\kappa_0}+\varepsilon_{\kappa_0}^3)$\; with \;$0<\varepsilon_{\kappa_0}^2<\varepsilon_{\kappa_0}$\;and \;$\varepsilon_{\kappa_0}^3>0$ \;some small real numbers. Finally, by setting 
$$
A_{\kappa_0}:=
A\Big(\lambda(\kappa_0), \kappa_0\Big),\; \text{and
}\;\;\mathcal{O}:=\bigcup_{\kappa_0\neq 0} \;A_{\kappa_0}\times [L_{\kappa_0}, L_{\kappa_0}+\varepsilon_{\kappa_0}^3),
$$
we have  that \;$\mathcal{O}$\; is an open neighborhood of the rectangular class \;$\tilde \Pi(\lambda(\kappa_0), \kappa_0)$\; for every \;$\kappa_0\neq 0$\; in the moduli space of conformal tori, as the rectangular classes lies on the boundary of the moduli space. 
\end{proof}

\begin{Rem}\label{L= g_3^2}
The Taylor series of $\;L(\tilde \lambda) := L\Big(\tilde f^n_{\tilde \lambda,  \kappa_0(\tilde \lambda)}\Big)\;$ gives:

$$L(\tilde \lambda) = L(0) + \frac{1}{2}\delta^2 L\big(f^b\big)(\hat \varphi, \hat \varphi) \tilde \lambda^2 + \text{higher order terms},$$
i.e., 
Therefore, we obtain $\;\tilde \lambda \sim \sqrt{L(\tilde \lambda) - L(0)}$. 
\end{Rem}

For integers $\;n\;$ the so constructed family $\;\tilde f^n_{\tilde \lambda, \overline\kappa_0(\tilde \lambda)}\;$ does not give appropriate candidates for constrained Willmore minimizers. By Lemma \ref{mainobservation} the limit Lagrange multiplier gives the derivative of the WIllmore energy at the homogeneous torus. It turns out that this limit Lagrange multiplier can be further decreased  when considering the associated family of constrained Willmore Hopf cylinders. 
For the surfaces in the associated family it is harder to compute the limit Lagrange multiplier than in the Hopf case, but the normal variation at a  homogenous torus (parametrized as a $\;(m,k)$-equivariant surface) turns out to be closely related, as we show in Proposition \ref{normalvariation}.  The limit Lagrange multiplier is the value for which the stability operator of  a penalized Willmore stability operator has a non-trivial kernel. As computed in \cite[Section 3]{HelNdi1} (and also in Appendix \ref{stability}), the normal variation uniquely determines the limit Lagrange multiplier. For the constrained Willmore Hopf tori $\;\tilde f^n_{\lambda, \tilde  \kappa_0}\;$ the normal variation at homogenous tori is specified in the following proposition.

\begin{Pro} The normal variation of the ($1$-parameter) sub family $f_t,$ smooth in $t \sim 0, $ of  $f_{\lambda, \kappa_0}^n$ satisfying $\bigPi^2\big( f_t\big) = const.$ with $f_0 = f^n_{\lambda(\kappa_0), \kappa_0}$ is given by 

$$\varphi^b = \sin\big(n\tfrac{x}{rs}\big) \vec{n}^b,$$

(up to invariance) where $f_{\lambda(\kappa_0), \kappa_0}^n$ is $f^b,$  $b = \tfrac{s}{r}$ and $\kappa_0=\tfrac{s^2-r^2}{rs},$ parametrized as a  $(1,1)$-torus

$$f^{b}: \C/ \big(2 \pi i \Z + 2 \pi( i r^2 + rs)\Z\big)\longrightarrow S^3, \quad f^b(x,y) =\begin{pmatrix} r e^{i \big(y+\tfrac{r x}{s}\big)}, &s e^{i \big(y-\tfrac{sx}{r}\big)}\end{pmatrix}$$

and $x$ is the profile curve parameter and $\vec{n}_b$ is the normal vector of $f^b$.\\

\end{Pro}

\begin{Rem}
For the Clifford torus $f^1$ parametrized as Hopf torus, the conformal type is $\tfrac{1}{2}(A, L) = (\pi, \pi).$ This is equivalent to $\;(a,b) = (0,1)\;$ after a $\;SL(2, \Z)$-transformation $\;T\;$ of the upper half plane. We adjust the projection $\;\bigPi: \text{ Imm }(T^2) \longrightarrow \H\;$  such that $\;\bigPi(f^1) = (0,1).$ The condition that $\;\dot b= 0\;$ corresponds  to $\;\dot A= 0\;$ at $\;(\pi, \pi)\;$ under this transformation. 
\end{Rem}
\begin{proof}
At $\;f^n_{0, 0}\;$ which is the Clifford torus (with intrinsic period $n$), we obtain that $\;\mu = n^2- \tfrac{1}{2} = \tilde \alpha\;$ is the limit Lagrange multiplier. 
Since the map $\;\tilde \bigPi,\;$ defined in Theorem \ref{openneigh}, covers an open neighborhood, we can take a lift of $\;\tfrac{1}{2}(A,L) \sim (\pi, \pi)\;$ under this covering and obtain  a family of surfaces parametrized by its conformal class and verifying the regularity assumptions needed for Lemma \ref{mainobservation}, see Remark \ref{L= g_3^2}. Thus the normal variation of $\;f_t\;$ at $\;t=0\;$ lies in the kernel of $\;\delta^2 \mathcal W_{ \alpha}\big(f^1\big),\;$ with $\;\alpha = 4 \pi^2 \tilde \alpha\;$, and depends only on the curve parameter. The relevant kernel elements of $\;\delta^2 \mathcal W_{\alpha}\big(f^1\big)\;$ can be computed analogously to the proof of Lemma \ref{kernvonalpha} (the $c=1$ case) and lie in the span of the vectors $\;\varphi_1 =  \sin(2nx) \vec{n}^1\;$ and $\;\tilde \varphi_1 = \cos(2nx) \vec{n}^1\;$. 
Thus  up to invariance, we can choose, without loss of generality, $\;\varphi_1\;$ to be the normal variation of  the family $\;f_t\;$ at $\;f_0 = f^n_{0, 0}\;$ the Clifford torus, proving the statement in this case. \\

 For the homogenous tori $\;f^b\;$ the statement holds, since $\;f^n_{\lambda, \kappa_0}\;$ is a smooth $2$-dimensional family of surfaces and the normal variation only depends on the curve parameter: consider the Fourier decomposition of the space of normal variations at $f^b$. The stability computations in Appendix \ref{stability} and \cite{KuwertLorenz} shows that normal variations to different Fourier modes are perpendicular with respect to $\delta^2\mathcal W_{\alpha, \beta^b}.$
Therefore,  there exist for every $\;b\;$ (and uniquely determined $\;\beta^b \sim 0$) a unique $\;\alpha_b \sim (n^2- \tfrac{1}{2}) 4 \pi^2\;$such that $\;\delta^2 \mathcal W_{\alpha_b, \beta^b}\big(f^{b}\big) $ has a kernel, and subspace of the kernel which depends only on the curve parameter is given by $\;\varphi_1^b = \sin(n\tfrac{x}{sr}) \vec{n}^b\;$ up to invariance, where $\;\vec{n}^b\;$ is the normal vector of the homogenous torus $\;f^b$. 

\end{proof}
 \begin{Pro}\label{Hopfnormal}
Let $f^{c(b)}_{(a,b)}$ be a family of constrained Willmore Hopf cylinders with conformal type $(a,b)$ with the same regularity assumptions as in Lemma \ref{mainobservation} such that $f^{c(b)}_{(0,b)}$ is the homogenous torus $f^b$ with
limit monodromy $M= \tfrac{1}{2 c(b)} \pi i$ as $a\longrightarrow 0$. Let $f_t:= f^{c(b(t))}_{(a(t), b(t))}$ be a smooth sub family with $a(0) = 0,$ $b(0) = b$ such that $a_0 := \lim_{t \rightarrow 0}\tfrac{\dot{a}(t)}{\sqrt{a(t)} }$ exists. Then for $b = \tfrac{s}{r}$ there exist $d, e \in \R$ such that the normal variation $\dot f:= \big(\del_{t}f_t|_{t=0}\big) ^\perp$ is given by 

$$\Phi = d^b \sin\big(c(b)\tfrac{x}{rs}\big) \vec{n}^b + e^b.$$

 \end{Pro}
 \begin{proof}
By the conformal type $(a,b)$ of a cylinder we mean the lattice $\Gamma$ generated by the the vectors $2 \pi$ and $\tfrac{1}{2}A + \tfrac{1}{2}L i $ with length $L = \tfrac{1}{c(b(t))} l$, where $l$ is the length of the arc length parametrized curve after one period of the curvature, and the ``enclosed'' oriented area $A = \tfrac{1}{c(b(t))}\int_0^{l} \kappa ds$ mod $4 \pi$ via Gau\ss-Bonnet. Moreover, we denote by $\dot{()}$ the derivative w.r.t. the parameter $t$ at $t=0.$ 

With the same computation as in the previous proposition it can be shown that for a family of $n$-lobed constrained Willmore Hopf tori converge to the $k$-times covered homogenous torus with $\dot b = 0$ the corresponding normal variation at $t=0$ is given by 
$$\Phi= \sin\big(\tfrac{nx}{k r s}\big)\vec{n}^{b}.$$

Thus by the continuous dependence of the family on the monodromy $c(b(t))$ we obtain that any family of (not necessarily compact) constrained Willmore Hopf cylinders deforming the homogenous torus with monodromy $M= \tfrac{1}{2 c(b)} \pi i$ and $\dot b= 0$ has normal variation at the $a=0$ given by 

$$\Phi = \sin\big(c(b)\tfrac{x}{ r s}\big)\vec{n}^{ b}$$

for a real number $c(b) \in \R.$ If $\dot b \neq0$  and $a(t) \neq 0$ we can split the normal variation into two components:

$$\dot f_t = \tfrac{1}{2}\tfrac{\dot a(t)}{\sqrt{a(t)}} \del_{\sqrt{a}} f^{c(b)}_{(a,b)} + \dot b(t)\del_b f^{c(b)}_{(a,b)}.$$

Because the surfaces $\;f^{c(b)}_{0,b}\;$ are homogenous, we can compute that $\;\big(\del_b f^{c(b)}_{(0,b)}\big)^\perp= \tilde e^b \vec{n}^b\;$ is constant hence we obtain for $t \longrightarrow 0$

$$\big(\dot f\big)^\perp = \big(a_0\sin(c(b) x)   + \dot b\tilde e^b\big) \vec{n}^b$$
proving the statement with $d^b =a_0$ and $e^b  = \dot b\tilde e^b.$ 
 \end{proof}

\section{$(1,2)$-Equivariant surfaces associated to constrained Willmore Hopf cylinders.}\label{12} $\ $\\
The stability computations leading to Lemma \ref{kernvonalpha} indicates that the candidates for constrained Willmore minimizers should have $(1,2)$-symmetry. Thus we construct in the following $2$-parameter families of $(1,2)$-equivariant tori deforming the Clifford torus whose projection into the space of conformal tori cover an open neighborhood. The crucial property of these candidates is that the limit Lagrange multiplier (as the surfaces converge to the Clifford torus) is $\;\alpha^1$, i.e., the maximum $\;\alpha>0\;$ for which $ \;\mathcal W_\alpha\;$ is stable at $\;f^1$.\\
 
For our candidates we use the Ansatz that they lie in the associated family of constrained Willmore Hopf cylinders, i.e., the conformal Hopf differential of the $(1,2)$-surface is given by 

$$4q_{(1,2)} =  \big(\kappa  + \sqrt{G}i\big)e^{2i\theta},$$

for $\;G,\; \theta \in  \R_+\;$ and $\;\kappa: \R \longrightarrow \R\;$ satisfying

\begin{equation}\kappa'' + \tfrac{1}{2}\kappa^3 + (\mu + \tfrac{G}{2}) \kappa + \lambda = 0,
\end{equation}

which is exactly the (constrained) elastic curve equation for curves on a round $S^2$ considered before. The real constants $\;\mu , \;\lambda \in \R\;$ are the Lagrange multipliers of the constrained Willmore Hopf surface. The corresponding Lagrange multipliers for the $(1,2)$-equivariant surface are given by \eqref{Lagrange} and are depending on $\;\;\lambda, \;\mu\;$ and $\;\theta$.\\

\subsection{Seifert fiber space} $\;$\\
We want to restrict to the $(1,2)$-equivariant case in the following, although the constructions below can be used to construct general $(m,n)$-tori (lying in the associated family of Hopf cylinders).
We consider $S^3 \subset \C^2$ with the equivalence relation

$$(z, w) \sim (\tilde z, \tilde w) \Leftrightarrow \exists \;\phi : (\tilde z, \tilde w) = \big(e^{i\phi} z,  e^{2i\phi}w \big).$$

\noindent
We can always choose a unique representative of $\;[(z,w)]\;$ of the form $\;\big(|z|, \tilde w\big) \in S^3,\;$ since for $\;z =|z|e^{i\phi}\;$ we have $\;\big(|z|e^{i\phi},w\big) \sim (|z|, e^{-2 i\phi} w ).\;$ The orbit space $\;S^3{/_\sim}\;$ is a topological $2$-sphere and the fibers of a point $\;(z,w) \in S^3 \subset \C^2\;$ is given by the curve

$$\varphi \longmapsto \big(e^{i\phi}z, e^{2i \phi}w\big).$$

 The trippel $\;F = \big(S^3, S^3/_{\sim}, \pi\big)\;$ defines a Seifert fiber space with one exceptional fiber over $\;[(0,1)],\;$ where $\;\pi: S^3 \longrightarrow S^3/_\sim\;$ is the projection map. 
 In the following we parametrize the regular set of $\;S^3/_\sim\;$ (which is a sphere $\;S^2 \setminus \{$one point$\}$ ) using polar coordinates:
 
$$[0,1] \times [0,1) \longrightarrow S^3{/_\sim} \quad (R, \phi) \longmapsto \big(R, \sqrt{1-R^2}e^{i \phi}\big).$$

 The round metric $\;g_{round}\;$ on $\;S^3\;$ induces a unique metric on (regular set of) the base space, such that $\;\pi\;$ becomes a Riemannian submersion. We denote this metric also by $\;g_{round}$. To be more specific, let $\;X, \;Y\;$ be local vector fields on $\;S^3/_{\sim}\;$ and $\;\hat X, \;\hat Y\;$ be their lifts to $\; TS^3\;$ w.r.t. $\;\pi.$ Then we define the metric on the base space to be
 
$$g_{round}(X,Y) = g_{round}\big((\hat X)^\perp, (\hat Y)^\perp\big),$$

where $\;()^\perp\;$ is the projection orthogonal to the fiber direction. In terms of the coordinates $\;(R, \phi)\;$ the metric on the base space is given by

$$g_{round}= \big(1-R^2\big)dr^2 + \frac{R^2(1-R^2)}{4-3R^2}d\phi^2.$$

With respect to the round metric the length of the fibers of $\;F\;$ at $\;(z, w) \in S^3 \subset \C^2\;$ is given by $\;L = 2 \pi l,\;$ where

$$l_{(z,w)} = \sqrt{|z^2| + 4|w^2|} = \sqrt{4-3R^2}.$$

Dividing the metrics $\;g_{round}\;$ point wise by $\;l^2,$ which is constant along every fiber,  yield new metrics $\;g_{(1,2)}\;$ on $\;S^3\;$ and on the base space $\;S^3/_{\sim},\;$ respectively, given by 

$$g_{(1,2)} = \tfrac{1}{4-3R^2}g_{round}.$$ 

With respect to $\;g_{(1,2)}\;$ all fibers have same length $\;2 \pi$. Let $\;B\;$ be the unit fiber direction with respect to $\;g_{(1,2)}$, then a bundle connection on $\;F\;$ is given by $\;\omega = g_{(1,2)}(B, .)\;$ and its curvature function be straightforwardly computed to be $\;\Omega = \tfrac{4}{\sqrt{4-3R^2}},\;$ see \cite[Proposition 3]{He1}.

Any closed curve 

$$\gamma : S^1 \longrightarrow S^3/_{\sim} \setminus \big\{[(1, 0)], [(0,1)]\big \}, x \longmapsto \big(R(x), \phi(x)\big)$$

 gives rise to an immersed equivariant 
torus by

$$f(x,y) = \big(e^{iy}R(x), e^{2iy} \sqrt{1-R^2(x)}e^{i\phi(x)}\big).$$

The torus is embedded, if and only if the curve is. To obtain a conformal parametrization of the surface, we need the profile curve to be arc length parametrized with respect to $g_{(1,2)}$ and take its horizontal lift  

$$\tilde \gamma(t) = \big(e^{iy(x)}R(x), e^{2iy(x)}\big(1-R(x)^2\big)e^{i\phi(x)}\big)$$

 to $\;S^3\;$ (i.e, $\;\tilde \gamma\;$ satisfy $\;\omega(\tilde \gamma') = 0\;$ and $\;\tilde \gamma\;$ is not necessarily closed). The conformal Hopf differential of the immersion is given by 
 
$$4q_{(1,2)} = \kappa_{(1,2)} + i \Omega.$$

where $\;\kappa_{(1,2)}\;$ is the geodesic curvature of $\;\gamma\;$ w.r.t $\;g_{(1,2)},\;$  see \cite[Proposition 3]{He1}.\\

\subsection{Construction of candidate surfaces}\label{proof}$\ $\\
For given conformal Hopf differential $\;q = e^{2i \theta} \big(\kappa + i \sqrt{G}\big)\;$ lying in the associated family of constrained Willmore Hopf solutions, we want to show how to obtain a constrained Willmore and $(1,2)$-equivariant cylinder and determine the closing conditions. 
Without loss of generality we always choose the $\;G>0\;$ in \eqref{elastic} possibly depending on $\;(\theta, \mu, \lambda, \kappa_0)\;$ such that the resulting $(1,2)$-surface is arc length parametrized with respect to $\;g_{(1,2)}.$
The curvature $\;\kappa\;$ of constrained Willmore Hopf cylinders depends on three parameters $\;\lambda,$ $\mu\;$ and $\;\kappa_0.$ It can be easily computed that the derivative of the equivariance type w.r.t. $\;\theta\;$  at the $(1,2)$-parametrized Clifford torus lying in the associated family of a $(1,1)$-parametrized homogenous torus (for $\;\theta = \arctan(2)- \arctan(b_{Hopf})\;$ and $\;\lambda_\theta \neq 0\;$ and  $\;\mu_\theta\neq 0)\;$ is non zero, thus there exist by implicit function theorem a function $\theta(\lambda, \mu, \kappa_0)\;$ such that $\;f_{\theta(\lambda, \mu, \kappa_0)}\;$ is $(1,2)$-symmetric. \\

Recall that the surfaces we are interested in lie in the associated family of constrained Willmore Hopf cylinders, i.e., the conformal Hopf differential of the $(1,2)$-equivariant surface is given by 

$$4q_{1,2} = 4q_{\theta} = \big(\kappa + i \sqrt{G}\big)e^{2i\theta}.$$

To construct the profile curve $\;\gamma\;$ of the $(1,2)$-equivariant surface in $\;S^3/_\sim,\;$ we first show that $\;\gamma\;$ is already uniquely determined up to isometries by 

$$ \Omega = 4\text{Im}(q_{1,2}) = \big(\cos(2\theta)\kappa +\sin(2\theta)\sqrt{G}\big).$$

On the other hand, recall that

$$\Omega = \tfrac{4}{\sqrt{4-3R^2}},$$ thus

$$R^2= \tfrac{4}{3}- \tfrac{16}{3}\tfrac{1}{\Omega^2}.$$

Further, we normalized our profile curve to be arclength parametrized. The round metric on $\;S^3\;$ induce a metric on $\;S^3/_{\sim}\;$ given by

$$g_{round} = (1-r^2)(dR)^2 + \frac{R^2\big(1-R^2\big)}{4-3R^2}\big(d\phi\big)^2.$$

Thus the arc length condition gives rise to a condition on $\;\phi'\;$ for $\;\gamma = \big(R(s), \phi(s)\big)\;$:

$$\big(1-R^2\big) \big(R' \big)^2 +  \frac{R^2\big(1-R^2\big)}{4-3R^2} \big(\phi'\big)^2 = \big(4-3R^2\big).$$

Therefore

\begin{equation}\label{angle}\phi = \pm \mathlarger{\mathlarger{\int_0^t}}\mathsmaller{\tfrac{1}{R}\sqrt{ \tfrac{\big(4-3R^2\big)^2}{1-R^2} -  (R')^2(4-3R^2) }dt.}
\end{equation}

The choice of initial value for $\;\phi\;$ corresponds to an isometry of the ambient space $\;S^3/_\sim\;$ and the choice of the sign corresponds to the orientation of the curve. Hence without loss of generality we choose $\;\phi(0) = 0\;$ and the ``$+$" sign.

\subsubsection{{\bf{Step 1: Existence of a $2$-parameter family of candidates}}}$\ $\\\label{2dfamily}In order to get closed curves, the necessary condition is that $\;R = \sqrt{\tfrac{4}{3}- \tfrac{16}{3}\tfrac{1}{\Omega^2}}\;$ is periodic. This holds automatically for $\;\Omega = \cos(2\theta) \kappa+\sin(2\theta)\sqrt{G},\;$ where $\;\kappa\;$ is a  wavelike solution of the elastic curve equation  ($D= g_3
^2 - 27g_2^
3<0$, see \cite[Lemma 2]{He2}). Moreover, in this case the closeness of the curve is equivalent to the angle function $\phi$, defined in \eqref{angle},  satisfying 

$$\Phi(g_2, g_3, \mu):= \phi_{(g_2, g_3, \mu)}(L) - \phi_{(g_2, g_3, \mu)}(0) = \tfrac{l}{n}2\pi, \quad \text{ for intergers } l,\;n.$$

We want to show that there exists a two parameter family of closed curves deforming the Clifford torus using implicit function theorem. Recall that  constrained elastic curves in $\;S^2\;$ can be obtained using Weierstrass elliptic functions, see Theorem \ref{curves}. 
By Lemma 2 of \cite{He2} we have that 

\begin{equation}
\label{kappa}
\kappa= \pm \sqrt{-8Re\big(\wp(x+x_0)\big) - \tfrac{2}{3}\big(\mu+\tfrac{G}{2}\Big)},
\end{equation}

solves the elastic curve equation for the parameters given by $\;g_2,\,$ $g_3\;$ and $\;\big(\mu+\tfrac{G}{2}\big),\;$ and with $\;x_0\;$ chosen such that $\;P_4\big(\kappa(0)\big) = 0.$\\

Again the homogenous tori appear in this description as limits of the generic solutions where the lattice $\;\Gamma\;$ (on which the $\;\wp\;$ function is defined) degenerates, i.e., when the discriminant given by $\;D = g_2^3-27g_3^2 \nearrow 0.\;$  Moreover, the limit solution $\;\kappa\;$ for $\;D \longrightarrow 0\;$ is constant. For a smooth family $\;f_t,\;$ $t \in [0, \varepsilon)\;$ of $(1,2)$-equivariant constrained Willmore tori deforming the Clifford torus (i.e., $\;f_0 = f^1\;$ in $(1,2)$-parametrization) we have $\;R_t' \longrightarrow 0\;$ and $\;R_t \longrightarrow \tfrac{1}{\sqrt{2}},\;$ as $\;t \longrightarrow 0.$ Therefore, we aim at solving
 $$\phi(L_\infty) -\phi(0) = 5 L_\infty   = \frac{l}{n}\pi,$$ 
 
 where $\;L_\infty\;$ is the limit period of $\;R_t\;$ as $\;t\longrightarrow 0,\;$ or equivalently the $(1,2)$-profile curve length of $\;f^1\;$ in $\;g_{(1,2)}\;$ metric. Therefore, we can compute the parameters $\;\hat g_2\;$ and $\;\hat g_3\;$ giving the parametrization of the Clifford torus $\;f^1\;$  (or piece of it) from $\;L_{\infty}\;$ and $\;\frac{l}{n}.$\\

\begin{Pro}\label{delmu}
For $\;g_2 \sim \hat g_2\;$ and $\;g_2 \sim \hat g_3\;$ there exist a continuous function $\;\mu_\theta(g_2, g_3)\;$ such that

$$\Phi\big(g_2,  g_3, \mu_\theta(g_2, g_3)\big) \equiv 2 \pi.$$

\end{Pro}

\begin{Rem}
\;
This implies the existence of a smooth $2$-dimensional family of $(1,2)$-equivariant constrained Willmore tori with extrinsic period $1$ by changing the parameters $\;g_2,\;$ $g_3\;$ and $\;\mu_\theta\;$ to $\;\tilde \lambda, \; \kappa_0\;$ and $\; \mu(\tilde \lambda, \tilde \kappa_0)\;$ as in Remark \ref{regularityM}. The same arguments yield $2$-parameter families of $n$-lobed constrained Willmore tori with $(1,2)$-symmetry.\end{Rem}
\begin{proof}
The map $\Phi$ is differentialble in $\mu_\theta$ and continuous in $g_2$ and $g_3.$
We want to show that $\tfrac{\del}{\del \mu_\theta}  \Phi(g_2, g_3, \mu_{\theta}) \neq 0$ at the simply wrapped Clifford torus. Then the assertion follows by a version of the implicit function theorem as in the proof of  Theorem \ref{smoothfamily2d}. At the Clifford torus we have  $R= \tfrac{1}{\sqrt{2}}$ and $R'=0$, therefore

\begin{equation}
\begin{split}\frac{\del }{\del \mu_\theta}\Phi(g_2, g_3, \mu_\theta)\mathlarger{|}_{R \equiv \tfrac{1}{\sqrt{2}}} = \mathlarger{ \int_0^{L_\infty}}  \frac{\del R}{\del \mu_\theta}\frac{4-5R^2}{R^2 (1-R^2)^{3/2}}\mathlarger{\mathlarger{|}}_{R\equiv \tfrac{1}{\sqrt{2}}} ds
 =  6 \sqrt{2}  \mathlarger{ \int_0^{L_\infty}}  \frac{\del R}{\del \mu_\theta} ds
\end{split}
\end{equation}

Further, we have 

$$\frac{\del R}{\del \mu_\theta} = \frac{5\sqrt{5}}{6} \frac{\del\Omega}{\del \mu_\theta}.$$ 

Thus it remain to show that $\frac{\del \Omega}{\del \mu_\theta} \neq 0.$ Since

$$\mu_\theta= \mu_{\theta}(g_2, g_3, \theta) = \cos\big(4\theta(g_2, g_3, \mu)\big)\mu + \sin\big(4\theta(g_2, g_3, \mu)\big)\lambda,$$ 

we have

$$\frac{\del \Omega}{\del \mu} = \frac{\del \Omega}{\del \mu_\theta} \frac{\del \mu_\theta}{\del \mu  }= \frac{\del \Omega}{\del \mu_\theta} \left(\cos\big(4 \theta(g_2, g_3, \mu)\big) + 4\frac{\del \theta(g_2, g_3, \mu)}{\del \mu}  \big(-\sin(4 \theta)\mu +  \cos(4\theta) \lambda\big)\right).$$

At the Clifford torus we have $\;\frac{\del}{\del \mu} \theta(g_2, g_3, \mu) >0, \;\lambda_\theta>0\;$ and $\;\cos(4 \theta)>0.\;$ Therefore, 

$$\frac{\del \Omega}{\del \mu}= const \frac{\del \Omega}{\del \mu_\theta}.$$

Now because

$$\Omega= \sin(2 \theta)\kappa +  \cos(2\theta)\sqrt{G}$$

we obtain 

$$\frac{\del \Omega}{\del \mu} =  \sin(2 \theta)\frac{\del \kappa}{\del \mu}\big(g_2, g_3, \mu\big) +  \cos(2\theta)\frac{\del}{\del \mu}  \sqrt{G(\theta, g_2, g_3, \mu)}+ \frac{\del \theta}{\del \mu}\big(g_2, g_3, \mu\big) \kappa_{1,2}.$$

At the $(1,2)$-parametrized Clifford torus, we have $\;\frac{\del }{\del \mu}G(\theta, g_2, g_3, \mu)=0,$\footnote{Since the torus remains a segment of the Clifford torus.} and $\; \theta(\mu)= \arctan(2)- \arctan\big(b_{Hopf}(\mu)\big),\;$ where $\;b_{Hopf}\;$ is the conformal type of the associated Hopf cylinder Moreover, $\;\lambda \neq 0,\;$ hence we can choose the sign of $\;\kappa\;$, i.e., $\;\kappa(0) > 0\;$ as before and obtain

$$\frac{\del \kappa}{\del \mu}=- \frac{2}{3\kappa(0)} <0.$$

A computation shows that the geodesic curvature of $\;f^1\;$ w.r.t. the $(1,2)$-metric is negative, further it can be computed that $\;\frac{\del \theta}{\del \mu} >0.\;$ We thus obtain

$$\frac{\del \Omega}{\del \mu} =- \sin(2 \theta) \frac{2}{3\kappa(0)} + \frac{\del \theta}{\del \mu} \kappa_{1,2} < 0.$$

\end{proof}
\subsubsection{{\bf Step 2: Candidates cover an open neighborhood of the moduli space of conformal tori}} \label{2dfamily} $\ $\\
We want to show that the family of surfaces constructed in the previous section cover an open neighborhood of rectangular conformal classes which are near but different from the square conformal class. Recall that the candidate family is parametrized by the parameters $\;(\kappa_0, \lambda).\;$\;  and denoted \;$f^n_{\lambda, \kappa_0}$. Analogously to Proposition \ref{openneigh} we have for every $\;\kappa_0\;$ an unique $\;\lambda(\kappa_0)\in \R_-\;$ such that $\;\big(\kappa_0, \lambda(\kappa_0)\big)\;$ corresponds to a homogenous torus denoted \;$f^n_{\lambda(\kappa_0), \kappa_0}$. Moreover, as in Proposition \ref{openneigh}, we change the parameter $\;\lambda\;$ to $\;\tilde \lambda = \lambda - \lambda(\kappa_0)\;$ in the following.\\

The conformal class of a $(1,2)$-equivariant surface is given by the length $\;L_{1,2}\;$ of its profile curve $\;\gamma\;$ w.r.t.~the $\;g_{(1,2)}\;$ metric and 

$$A_{1,2}= \int_{\mathcal C} \Omega_{1,2} vol_{1,2},$$

 where $\;\mathcal C\;$ is a $2$-chain with \;$\partial \mathcal C = \gamma_{1,2}.\;$ For every $\kappa_0\neq 0$, at the corresponding homogenous tori \;$f^n_{\lambda(\kappa_0), \kappa_0}$, we have by definition of \;$\;L_{1,2}\big(\gamma_{1,2}(\tilde \lambda, \kappa_0)\big)\;$ \;that

$$\frac{\partial L_{1,2}\big(\gamma_{1,2}(\tilde \lambda, \kappa_0)\big)}{\partial \kappa_0}{{|}}_{(\tilde \lambda, \kappa_0) = (0, \kappa_0)} \neq 0,$$
where \;$\gamma_{1,2}(\tilde \lambda, \kappa_0)$\; is the profile curve of \;$f^n_{\tilde \lambda+\lambda(k_0), \;\kappa_0}$. Thus \;there exists a real analytic map 
$$\widetilde\kappa_0: (-\tilde\varepsilon_{\kappa_0}, \tilde\varepsilon_{\kappa_0})\longrightarrow (\kappa_0-\tilde\varepsilon_{\kappa_0}^1, \kappa_0+\tilde\varepsilon_{\kappa_0}^1)$$

such that

 $$\;L_{1,2}\big( \gamma_{1,2}(\tilde \lambda, \kappa_0(\tilde \lambda)\big) =L_{1,2}\big( \gamma_{1,2}(0, \kappa_0(0)\big) ,\quad \forall\;\tilde\lambda\in (-\tilde\varepsilon_{\kappa_0}, \tilde\varepsilon_{\kappa_0})\;$$ with \;$\tilde\varepsilon_{k_0}$ and \;$\tilde \varepsilon_{k_0}^1$\; positive numbers and \;$\kappa(0)=\kappa_0$. Following the proof of Proposition \ref{openneigh}, since the function $\tilde \lambda \in (-\tilde\varepsilon_{\kappa_0}, \tilde\varepsilon_{\kappa_0})\longrightarrow  A_{1, 2}\Big(\gamma_{1, 2}(\tilde \lambda,\;  \widetilde\kappa_0(\tilde \lambda))\Big)$\; is real analytic,  there exists \;$k\geq 1$\; such that 
\begin{equation}\label{nonconstant1}
\frac{d^k A_{1, 2}\Big(\gamma_{1, 2}(\tilde \lambda,\;  \widetilde\kappa_0(\tilde \lambda))\Big)}{d\tilde \lambda^k}{{|}}_{\tilde  \lambda= 0} \neq 0, \;\;\text{and}\;\;\frac{d^i A_{1, 2}\Big(\gamma_{1, 2}(\tilde \lambda,\;  \widetilde\kappa_0(\tilde \lambda))\Big)}{d\tilde \lambda^i}{{|}}_{\tilde  \lambda= 0} = 0\;\;\text{for}\;\;1\leq i\leq k-1.
\end{equation}
or
\begin{equation}\label{constant1}
\;A_{1, 2}\Big(\gamma_{1, 2}(\tilde \lambda,\;  \widetilde\kappa_0(\tilde \lambda))\Big)=\;A_{1, 2}\Big(\gamma_{1, 2}(0,  \;\kappa_0)\Big)\;\;\; \forall\; \tilde \lambda \in (-\varepsilon_{\kappa_0}, \;\varepsilon_{\kappa_0}),
\end{equation}
where \;$\gamma_{1,2}(\tilde \lambda, \kappa_0(\tilde \lambda))$\; is the profile curve of \;$f^n_{\tilde \lambda+\lambda(k_0(\tilde \lambda)), \;\kappa_0}$. Let us show that  \eqref{constant1} can not happen.  Suppose  that \;$\;A_{1,2}\Big( \gamma_{1,2}\big(\tilde \lambda, \kappa_0(\tilde \lambda)\big)\Big)\;$\; is constant for \;$\tilde \lambda\sim 0$\; where $\tilde \lambda=0$\;corresponds to the\;$(1,2)$-parametrized homogenous torus  \;$f^n_{\lambda(\kappa_0), \;\kappa_0}$, namely
 $$
 \;A_{1, 2}\Big(\gamma_{1, 2}(\tilde \lambda,\;  \widetilde\kappa_0(\tilde \lambda))\Big)=\;A_{1, 2}\Big(\gamma_{1, 2}(0,  \;\kappa_0)\Big)\;\;\; \forall\; \tilde \lambda \in (-\varepsilon_{\kappa_0}, \varepsilon_{\kappa_0}).
 $$
 Then the family $\;\gamma_{1,2}\big(\tilde \lambda, \;\widetilde\kappa_0(\tilde \lambda)\big)\;$ would give rise to a family of constrained Willmore surfaces in the conformal class of the homogeneous tori \;$f^n_{\lambda(\kappa_0), \;\kappa_0}$\; deforming the homogenous tori \;$f^n_{\lambda(\kappa_0), \kappa_0}$\; but different up to invariance to them. Thus we obtain a non Moebius equivalent family of constrained Willmore surfaces with the same conformal class and same Willmore energy as the homogenous torus corresponding to the parameter \;$\kappa_0$. So for $\kappa_0\neq 0$ and $\kappa_0\sim 0$, we get a contradiction to the fact that homogenous tori are the unique constrained Willmore minimizers up to Moebius invariance, see \cite{NdiayeSchaetzle1}.  Hence, we have \eqref{nonconstant1} must hold implying the existence of a right-neighborhood \;$[0, \tilde\varepsilon_{\kappa_0}^2)$\; of \;$0$\; such that the image of \;$[0, \varepsilon_{\kappa_0}^2)$\; under the map \;$\tilde \lambda \in (-\varepsilon_{\kappa_0}, \varepsilon_{\kappa_0})\longrightarrow  A_{1, 2}\Big(\gamma_{1, 2}(\tilde \lambda,\; \kappa_0(\tilde \lambda))\Big)$\;is a right -neighborhood of \;$A_{1, 2}^{\kappa_0}:=A_{1, 2}\Big(\gamma_{1, 2}(0,  \kappa_0)\Big)$\; that we denote by \;$[A_{1, 2}^{\kappa_0}, \;A_{1, 2}^{\kappa_0}+\tilde\varepsilon_{\kappa_0}^3)$\; with \;$0<\tilde\varepsilon_{\kappa_0}^2<\tilde\varepsilon_{\kappa_0}$\;and \;$\tilde\varepsilon_{\kappa_0}^3>0$\; some small real numbers. Finally, setting 
$$
L_{1, 2}^{\kappa_0}=
L_{1, 2}\Big(\gamma_{1, 2}(0, \kappa_0)\Big),\; \text{and
}\;\;\mathcal{O}:=\bigcup_{\kappa_0\neq 0, \;\kappa_0\sim 0} \; [A_{1, 2}^{\kappa_0}, \;A_{1, 2}^{\kappa_0}+\tilde \varepsilon_{\kappa_0}^3)\times {L_{1, 2}^{\kappa_0}},
$$
we have that $\mathcal{O}$ 
is an open neighborhood of the rectangular class \;$(A_{1, 2}^{\kappa_0}, \;L_{1, 2}^{\kappa_0})$\; in the moduli space of tori.

\begin{Rem}Since the above constructed family of candidates cover an open neighborhood of 

$$(0, b_0) \; \in \; \big\{(a,b) \in \R^2 \; \mathlarger{|} a \geq 0, b\geq \sqrt{1-a^2} \big\} \quad b_0 \sim 1, b_0 \neq 1$$

 we can parametrize it by its conformal type $\;(a,b)\;$, $b\simeq 1, b\neq 1$, and $a\sim_b 0^+$\; instead and denote it by $\;f_{(a,b)}\;$ in the following. As before in the Hopf case, due to the degeneracy of the homogenous, the family $\;f_{(a,b)}\;$ is not smooth in $\;a\;$ but rather in $\;\sqrt{a},\;$ see Remark \ref{regularityM}.
\end{Rem}

\subsubsection{{\bf Step 3: Identification of the normal variation at $f^b$}} \label{convergence}$\ $\\
The family of constrained Willmore tori $\;f_{(a,b)}\;$ constructed here lie in the associated family of constrained  Willmore Hopf cylinders.  Proposition \ref{Hopfnormal} gives that the normal part of the variation $\;\dot f^{Hopf}:= \Big(\del_{\tilde \lambda} \tilde f^{Hopf}_{\big(\tilde \lambda,\overline\kappa_0(\tilde \lambda)\big)}\Big)^\perp\;$ is given by $\;\sin(c x) \vec{n}_{1,1}^{\tilde b}\;$ for some real constant $\;c.\;$ The following Lemma relate the normal variation of $f_{(a,b)}$ to the normal variation of $\;f^{Hopf}_{(\tilde \lambda,\overline \kappa_0(\tilde \lambda))}$\; at \;$f^b$. 

\begin{Lem}\label{normalvariation}
For $\;b \sim 1 \;$ fixed  and $\;a \sim_b 0^+,\;$ let $\;f_{(a,b)}\;$ be the family of $(1,2)$-equivariant constrained Willmore tori of conformal type $\;(a,b)\;$ lying in the associated family of certain constrained Willmore Hopf cylinders $\;f^{Hopf}_{(\lambda,\kappa_0)}.\;$ Then there ist a $\tilde c \in \R_{\geq 0}$ with

$$<\partial_{\sqrt{a}} f_{(a,b)}\mathlarger{|}_{a=0}, \vec{n}_{1,2}^{b} > \;= \;const \cdot \sin (\tilde c x) \vec{n}_{1,2}^{b},$$

where $\;\vec{n}_{1,2}^{b}\;$ is the normal vector of the homogenous torus with conformal class $(0,b)$ parametrized as a $(1,2)$-equivariant surface as in \eqref{12fb} with conformal
 Hopf differential 
 
 $$\;q_{1,2} = q_{1,1} e^{2 i \theta}.$$\\

\end{Lem}
\begin{Rem}
Note that both families of surfaces are homogenous for $\;a=0\;$, i.e., $\;a=0\;$ implies $\;\tilde \lambda = 0$\; or equivalently \;$\lambda=\lambda(\kappa_0)$.
\end{Rem}

\begin{proof}
Since we are considering equivariant tori and variations of these preserving the equivariance type, we can restrict ourselves without loss of generality to the variation of the the underlying profile curves. Recall that the conformal type of the surfaces are given by 

$$A_{1,2}= \int_{\mathcal C} \Omega_{1,2} vol_{1,2} \quad \text{and} \quad L_{1,2} = \int_{\gamma_{1,2}} ds_{1,2}.$$

Let $\;4q_{1,1}^t = \kappa_t + i\;$ be the conformal Hopf differential of the Hopf cylinders $\;f^{Hopf}_{\lambda(t), \kappa_0(t)}.\;$ Then the conformal Hopf differential of the associated $(1,2)$-tori are given by 

$$4q_{1,2} = \big(\cos(2\theta_t)\kappa_t -  \sin(2\theta_t)\big) + i \big(\sin(2\theta_t) \kappa_t +\cos(\theta_t)\big).$$

Note that the formula for the associated family given in Section \ref{Tr} is the same one as given in \eqref{12fb} up to scaling $\;T^2_b\;$ by a constant factor (constant in $\;x\;$ and $\;y$).

Since $\;b = const\;$ and the homogenous tori are degenerate, we have $$\;\dot A_{1,2}= 0 = \dot L_{1,2}.\;$$
at $a=0.$ This gives rise to the following
 
\begin{equation}\label{dotA}
\dot A_{1,2} = \int_{\gamma_{1,2}}  \Omega_{1,2} vol_{1,2} \big(\dot \gamma_{1,2}, \gamma'\big) = \int_{\gamma_{1,2}} \Omega_{1,2} g_{(1,2)} \big(\dot \gamma_{1,2}, \vec{n}_{1,2}^b\big) = 0,
\end{equation}
\begin{center}and\end{center}
\begin{equation}
\dot L_{1,2} = \int_{\gamma_{1,2}} g_{1,2} \big(\dot \gamma_{1,2}', \gamma'\big) ds= -\int_{\gamma_{1,2}} \hat \kappa_{1,2} g_{1,2} \big(\dot \gamma_{1,2}, \vec{n}_{1,2}^b\big) = 0.
\end{equation}

 Since for $\;a = 0\;$ the family $\;f_{(0,b)}\;$ are homogenous, the corresponding Hopf cylinders are homogenous too. Thus there exist a real numbers $\;c\;$ and $\;d\;$ such that  normal variation $\;\dot f^{Hopf}:=\big(\del_\lambda f^{Hopf}_{\lambda, \kappa_0}\big)^\perp\;$ of the Hopf cylinders $\;f^{Hopf}_{\lambda, \kappa_0}\;$ are given by Proposition \ref{Hopfnormal} to be
 
 $$ <\dot f^{Hopf}, \vec{n}_{1,1}^{\tilde b}>\;=\; d  + \sin (c x).$$
 
 From this we obtain that 
 
 $$<\dot f^{Hopf}, \vec{n}_{1,1}^{\tilde b}>'' \;=\; <\dot \gamma_{1,1}, \vec{n}_{1,1}^{\tilde b}>'' \;=\; -c^2 <\dot f^{Hopf}, \vec{n}_{1,1}^{\tilde b}> + c^2 d.$$
 
Since all Hopf tori are arclength parametrized, another  computation shows that 

 $$<\dot \gamma_{1,1}, \vec{n}_{1,1}^{\tilde b}>'' \;=\; \dot \kappa + ( -1-\kappa_0^2) <\dot \gamma_{1,1}, \vec{n}_{1,1}^{\tilde b}>,$$
 
 hence
 
 \begin{equation} \label{dotkappa1}\dot \kappa = \big(-c^2 + \kappa_0^2 +1\big)\sin(c x) + \big(\kappa^2_0 +1\big) d.
 \end{equation}
 
 For $\;c^2= \kappa_0^2 + 1\;$ we obtain that the variation has extrinsic period $1$  and the resulting surfaces (including the associated $(1,2)$-equivariant ones) are all homogenous.  At the $(1,2)$-parametrized homogenous torus, we have further that 
 
 $$\;\vec{n}_{1,2}^b|_{a=0} = \tfrac{\del }{\del R}\;$$ and thus

$$g_{1,2}\big(\dot \gamma_{1,2}, \vec{n}_{1,2}^b\big) = \dot R = const \cdot \dot \Omega = const \cdot \big(\sin(2\theta)\dot \kappa + 2 \dot \theta \hat \kappa_{1,2}\big),$$

where 

$$\;\hat \kappa_{1,2}:= \cos(2\theta) \kappa_0 -\sin(2\theta)\;$$

 for $\;\tilde \lambda=0.\;$ Since $\;\dot A_{1,2} = 0\;$  we obtain that 
 
 $$\;\int_{\gamma_{1,2}} g_{(1,2)}\big(\dot \gamma_{1,2}, \vec{n}_{1,2}^b\big) = 0\;$$ and therefore 

\begin{equation}\label{dotRkappa}
\begin{split}
\sin(2 \theta)\big(\kappa_0^2+1\big)&d = - 2\dot \theta \hat \kappa_{1,2}\\
\;\\
&\text{and}\\
\;
\quad \dot R = const \cdot \sin(&2\theta) \big(c^2 - \kappa_0^2 -1\big) \sin(c x).
\end{split}
\end{equation}

This proves the Lemma.

\end{proof}
The normal variation of $f_{(a,b)}$ at $f^b$ depends only on the profile curve parameter. Moreover, the intrinsic period of the normal variation is at most the intrinsic period of profile curve of $f_{(a,b)}$. When $f^b$ is parametrized as a $(1,2)$-equivariant torus \eqref{12fb} we therefore obtain
\begin{Pro}\label{Convergence}
A family of $\;n$-lobed $(1,2)$-equivariant constrained Willmore tori $\;\hat f^n_{\lambda(t), \kappa_0(t)}\;$ deforming the homogenous torus given by \eqref{12fb} with $\;\dot A_{1,2} = \dot L_{1,2} = 0\;$ has normal variation given by
$$<\partial_{\sqrt{a}} f_{(a,b)}\mathlarger{|}_{a=0}, \vec{n}_{1,2}^{b} > \;=  \sin \left (\frac{r^2  + 4s^2}{2 rs} n x \right),$$

up to invariance. 

\end{Pro}

\subsubsection{{\bf Step 4: The Lagrange multiplier converge from below}}$\ $\\
We have so far constructed a $2$-parameter family $f_{(a,b)}$ of constrained Willmore tori covering an open set of the moduli space of conformal tori. Now we want to show that their $\;\bigPi^1$-Lagrange multiplier is converging from below as the conformal class $\;(a, b) \longrightarrow (0,b)\;$ for a fixed $\;b \sim1, b \neq 1.\;$ To do so, we show that Lagrange multiplier function $\;\alpha(t)\;$ of the subfamily $\;f^t\;$ of the candidates with conformal class $\;(t^2, b)\;$  has a maximum at $\;t=0.\;$  Thus in a first step we show that $\;\alpha(t)\;$ has a critical point at $\;t=0.$\\

\begin{Lem}\label{Ldottheta}
For $\; b \sim1\;$ fixed let $\;f^t = f_{(t^2, b)}\;$ be a sub family of candidate surfaces with $f_{(0, b)} = f^b.\;$ Then the corresponding Lagrange multipliers satisfy 
$$\dot \alpha = \dot \beta =0.$$

\end{Lem}
\begin{proof}

We first observe that 
$\;\dot \beta = 0\;$ since

$$\dot \beta = \del_{t}\del_b\mathlarger{|}_{t=0} \mathcal W\big(f^t\big) = \del_b \del_{t}\mathlarger{|}_{t=0} \mathcal W\big(f^t\big) = \big(2 t\del_b\alpha^b\big)\mathlarger{|}_{a=0} = 0.$$

Therefore, $\;\dot \alpha = 0\;$ is equivalent to $\;\dot\mu_\theta = \dot \lambda_\theta = 0.\;$
Since the candidate surfaces lie in the associated family of constrained Willmore Hopf surfaces, the conformal Hopf differential of $\;f^t\;$ is given by $$\;q^t= \big( \kappa^t(x) + i\big)e^{2i \theta^t}.$$

The candidates surfaces are all constrained Willmore and thus satisfies the Euler Lagrange equation for all $\;t\;$ we have by \eqref{EL}:

\begin{equation}\label{ELI}
(q^t)'' + 8 \big(|q^t|^2+C^t\big)q^t - 8\xi^t q^t = 2 \Re\big(q^t (- \mu^t_{\theta^t}+ i \lambda^t_{\theta^t})\big), 
\end{equation}

with $\;\mu^t_{\theta^t},\;$ $\lambda^t_{\theta^t},\;$ $C^t\;$ and $\;\xi^t\;$ being the parameters of the $(1,2)$-equivariant surfaces given by \eqref{Lagrange} and $\;\mu^t,\;$ $\lambda^t,\;$ $C_0^t\;$ and $\;\xi^t_0 \;$ being the parameters of the associated Hopf surfaces. To abbreviate the notations we omit $\;t\;$ whenever $\; t=0\; ,$ e.g., $\;\mu_{\theta} = \mu^0_{\theta^0},\;$ $\lambda_{\theta} = \lambda^0_{\theta^0}.\;$  Moreover, let $\;\dot{()}\;$ and $\;\ddot{()}\;$ denote the derivatives w.r.t. $\;t\;$ evaluated at $\;t=0$.\\

Differentiating equation \ref{ELI} by $\;t\;$ at $\;t=0\;$ once the real and imaginary parts yields the following two equations

\begin{equation}
\begin{split}
&\dot q''_1 + 24\dot q_1 q_1^2 + 8\dot q_1 q_2^2 +16\dot q_2 q_1q_2 +  8C\dot q_1 + 8\dot Cq_1 + 8 \xi \dot q_2 + 8 \dot \xi q_2 +  2 \dot \mu_\theta q_1  + 2  \mu_\theta \dot q _1+ 2\dot \lambda_\theta q_2 + 2 \lambda_\theta \dot q_2=0,\\
&\dot q''_2 + 24\dot q_2 q_2^2 + 8\dot q_2 q_1^2 +16\dot q_1 q_1q_2 +  8C\dot q_2 + 8\dot Cq_2 - 8 \xi \dot q_1- 8 \dot \xi q_1=0,
\end{split}
\end{equation}

with $\;q_1:= \Re(q^0),\;$ $q_2:= \Im(q^0),\;$ $\dot q_1 := \frac{d}{dt}|_{t=0} \Re (q^t)\;$ and $\;\dot q_2 := \frac{d}{dt}|_{t=0} \Im (q^t).\;$ Proposition \ref{convergence} gives

\begin{equation}\label{dottheta}
\dot q_2 = c_1 \sin(c_2 x) \quad \text{ and } \quad \dot q_1 = c_1 \sin (c_2x) + c_4 \dot \theta
\end{equation}

with  non zero real constants $\;c_1,\;$ $c_2\;$ and $\;c_4 <0.$ 
For the associated family of constrained Willmore Hopf surfaces we can compute (together with \eqref{Lagrange}) that 

$$C = C_0 +  \tfrac{1}{8} \Re\big((e^{4i \theta} - 1)(- \mu - i \lambda)\big), \quad \text{ and } \quad \xi^t =c_6 \kappa  + \tfrac{1}{8} \Im \big((e^{4i \theta} - 1)(- \mu - i \lambda)\big)  $$

where we have used, that $\;\xi_0= c_6 \kappa \;$ at Hopf cylinders with $\;c_6>0.\;$ Hence,

$$8\dot C=-\dot  \mu_\theta + \dot \mu , \quad 8\dot \xi = c_6 \dot \kappa- \dot \lambda_\theta + \dot \lambda  \quad \text{ and } \quad \xi^0 = const.$$

Therefore, we can split \eqref{dottheta} into $\;\sin(c_2x)\;$ and a constant part. And since $\;1 \perp \sin(c_2 x)\;$, both needs to vanish for \eqref{dottheta} to hold. The constant part of  \eqref{dottheta} gives rise to the following equations 

\begin{equation}\label{dottheta3}
\begin{split}
 &c_4 \dot \theta \Big(24 q_1^2 + 8 q_2^2 +  8C  + 2  \mu_\theta + c_6 \tfrac{1}{2\sin(2 \theta)} q_1q_2\Big)+ \dot  \mu_\theta q_1 + \dot \mu q_1    + \dot \lambda_\theta q_2 + \dot \lambda q_2  =0\\
&c_4\dot \theta \Big(16 q_1q_2- 8 \xi^0 - c_6 \tfrac{1}{2\sin(2 \theta)}q_1^2\Big) -   \dot \mu_ \theta q_2 +  \dot \mu q_2   +  \dot  \lambda_\theta q_1 - \dot \lambda q _1=0.
\end{split}
\end{equation}

It can be computed that  

 $$c \beta =  - \mu_\theta q_1   -  \lambda_\theta q_2 \quad \text{ and } \quad c \dot \beta = - \dot  \mu_\theta q_1   -  \dot \lambda_\theta q_2  =0.$$

Therefore if $\;\dot \theta = 0\;$ we would obtain

\begin{equation}
\begin{split}
& \dot \mu q_1 + \dot \lambda q_2  =0\\
  &\dot \mu_\theta q_1 + \dot \lambda_\theta q_2  =0\\
 -  & \dot \mu_ \theta q_2 +  \dot \mu q_2  +  \dot  \lambda_\theta q_1 - \dot \lambda q _1=0,
\end{split}
\end{equation}

from which we can immediately deduce that $\;\dot \mu_\theta = \dot \mu\;$ and $\;\dot \lambda_\theta = \dot \lambda.\;$  Moreover, with 

\begin{equation}\label{dotlambdamutheta2}
\begin{split}
 \dot \mu  &= \cos(4\theta) \dot \mu_\theta + \sin(4 \theta) \dot \lambda_\theta  + 4\dot \theta \lambda_\theta\\
 \dot \lambda  &= -\sin(4\theta) \dot \mu_\theta + \cos(4 \theta) \dot \lambda_\theta - 4\dot \theta \mu_\theta
\end{split}
\end{equation} 

and $\;\theta \in (0, \tfrac{\pi}{2})\;$ this is equivalent to  $\;\dot \mu_\theta = \dot \lambda_\theta= 0.\;$ It thus remains to show $\;\dot \theta=0.\;$ Inserting \eqref{dotlambdamutheta2} into \eqref{dottheta3} we obtain  together with $\;\dot \beta= 0\;$ that 

\begin{equation}
\begin{split}
& \dot \theta \left [24  c_4q_1^2 + 8 c_4q_2^2 +  8C c_4  + 2  c_4 \mu_\theta + 
 c_6 \tfrac{1}{2\sin(2 \theta)}  q_1q_2 \right ]\\ 
& + \sin(4 \theta) \dot \lambda_\theta q_1- \sin(4\theta) \dot \mu_\theta q_2   =0\\
&\dot \theta \left[ 16 c_4 q_1q_2- 8 c_4 \xi^0  + 4 \mu_\theta q_1 +  4 \lambda_\theta q_2 
- c_6 \tfrac{1}{2\sin(2 \theta)} q_1^2\right ]\\ 
&- 2 \sin^2(2\theta) \dot \mu_ \theta q_2  
   + 2\sin^2(2 \theta) \dot \lambda_\theta q_1 =0.
\end{split}
\end{equation}

Multiplying the second equation by $\;\frac{\sin(4\theta)}{2\sin^2(2\theta)} = \frac{\cos(2 \theta)}{\sin (2\theta)}\;$ and then subtracting both equations yields either

\begin{equation}\label{c4}
\begin{split}
&24  c_4q_1^2 + 8 c_4q_2^2 +  8C c_4  + 2  c_4 \mu_\theta  
 + c_6 \frac{1}{\sin^2(2 \theta)} \hat \kappa_0 q_1 \\- &\frac{\cos(2 \theta)}{\sin(2 \theta)}\Big(16 c_4 q_1q_2- 8 c_4 \xi^0  + 4 \mu_\theta q_1 +  4 \lambda_\theta q_2 \Big)=0,
\end{split}
\end{equation}

or 

$$\dot \theta = 0.$$

Thus in order to show  $\;\dot \theta = 0\;$ we need to show that \eqref{c4} does not hold. We show this for $\;f^0\;$ being the Clifford torus. Then  \eqref{c4}  cannot hold by continuity of the candidates for all homogenous tori close to the Clifford torus. Due to the Clifford torus being Willmore we have 

$$- 2 \Re\big(q (- \mu_\theta + i \lambda_\theta)\big) = 2 \lambda_\theta q_2+ 2 \mu_ \theta q_1 = 0$$

 and since the Clifford torus is minimal in $\;S^3\;$ we have moreover 
 
 $$\xi^0 = const \cdot H = 0.$$ 
 
 Thus  \eqref{c4} reduces to 
 
 $$24  c_4q_1^2 + 8 c_4q_2^2 +  8C c_4  + 2  c_4 \mu_\theta  + c_6 \frac{1}{\sin^2(2 \theta)} \hat \kappa_0 q_1- \frac{\cos(2 \theta)}{\sin(2 \theta)}16 c_4 q_1q_2=0,$$
 
The Euler-Lagrange equation \eqref{ELI} for the Clifford torus gives 

$$q_1^2 + q_2^2 + C = 0$$ 

and together with $$\;q_1- \frac{\cos(2 \theta)}{\sin(2 \theta)}\quad\text{ and }\quad q_2 = - \tfrac{1}{\sin(2 \theta)}\;$$ we obtain

$$- 16  c_4q_1\tfrac{1}{\sin(2 \theta)}  + 2  c_4 \mu_\theta  + c_6 \frac{1}{\sin^2(2 \theta)} \hat \kappa_0 q_1= 0$$

Now since an easy computation shows that $\;c_4 < 0, \;$ $\;c_6>0\;$ and $\;\mu_\theta, \;\lambda_\theta >0,\;$ $\;q_1 <0,\;$ $\;q_2>0\;$ and $\;\hat \kappa_0>0\;$ we obtain that all summands are actually strictly negative leading to a contradiction.\\

\end{proof}

Now we show that $\;\alpha(t)\;$ indeed has a maximum at $\;t=0$.

\begin{Pro}
For $\; b \sim1\;$ fixed let $\;f^t = f_{(t^2, b)}\;$ be a sub family of candidate surfaces with $\;f_{(0, b)} = f^b.\;$ Further, let $\;\alpha(t)\;$ and $\;\beta(t)\;$ be the corresponding Lagrange multipliers. Then we have

$$\ddot \beta < 0 \quad \text{ and } \quad \ddot \alpha <0.$$  

\end{Pro}
\begin{proof}  
Using similar computations as in Lemma \ref{Ldottheta} can show that 

$$\ddot \beta = \del_b \alpha^b.$$

Since $$\;\alpha^{\tfrac{1}{b}} = \alpha^b,\;$$

 we obtain that $\;\ddot \beta =0\;$ at $\;b=1.\;$ On the other hand, the stability computations in Appendix \ref{stability} shows that $\;\alpha^b\;$ is decreasing in $\;b\;$ for $\;b>1.\;$ Therefore, 

$$\ddot \beta < 0$$ 

for $\;b >1\;$ due to the real analyticity of the candidate surfaces. As in the proof of Lemma \ref{Ldottheta} it is easy to compute that there exists a positive constant $\;c\;$ with

\begin{equation}\label{ddotb}
c \ddot \beta = - \ddot \mu_\theta q_1 - \ddot \lambda_\theta q_2 \leq0.
\end{equation}

Similar computations show that there is a other positive constant $\;\tilde c\;$ such that 

$$\tilde c \ddot \alpha =  \ddot \lambda_\theta q_1 - \ddot \mu _\theta q_2.$$
Therefore \eqref{ddotb} implies,

$$\tilde c \ddot \alpha  \leq - \tfrac{1}{q_2} \ddot \mu_\theta (q_1^2 + q_2^2).$$

It thus remains to show $\;\ddot \mu_\theta >0.\;$
The candidate surfaces constructed in the previous sections are associated to constrained Willmore Hopf cylinders, i.e., constrained elastic curves in $\;S^2,\;$ with geodesic curvature $\;\kappa\;$ being wavelike. Therefore the discriminant $\;D(t)= g^3_2(t) - 27g^2_3(t)\;$ of the corresponding Weierstrass $\;\wp$-functions is non-positive and has a maximum at homogenous solutions, i.e.,

$$\dot D = 0 \quad \text{ and } \quad \ddot D<0$$

The candidate surfaces all satisfy the closing condition:

$$\Phi\big(\mu_\theta(t), g_2(t), g_3(t)\big) \equiv 2\pi.$$

Replacing the parameters $\;g_2,\; g_3\;$ by $\; \mathsmaller{\sqrt{-D}}\;$ and $\;L_{1,2}\;$ we obtain

$$\Phi\big(\mu_\theta(t), \mathsmaller{\sqrt{-D}}(t), L_{1,2}(t)\big)\equiv 2\pi.$$

Note that for candidate surfaces $\;\Phi\;$ is continuous in $\;\mu_\theta,\;$ $\mathsmaller{\sqrt{-D}}\;$ and $\;L_{1,2}\;$ is real analytic in its parameters, since at $\;D=0\;$ we have that $\;\mathsmaller{\sqrt{-D.}}\;$ has the correct singularity.  By Lemma \ref{Ldottheta}

 $$\dot \theta = \dot \mu_ \theta = \dot \lambda_\theta = 0.$$

Moreover

$$\dot L_{1,2} = 0,$$  

and thus 
 $\;\dot \Phi = 0\;$ is equivalent to $\;\del_{\mathsmaller{\sqrt{-D}}}\Phi\mathlarger{|}_{{\mathsmaller{\sqrt{-D}}}=0} = 0.\;$ Therefore, the second derivative of $\;\Phi\;$ is computed to be:
 
$$\ddot\Phi= \ddot \mu_\theta \del_{\mu_\theta} \Phi + \big(\dot {\mathsmaller{\sqrt{-D} }}\del^2_{\mathsmaller{\sqrt{-D} \sqrt{-D}}} \Phi\big)\mathlarger{|}_{{\mathsmaller{\sqrt{-D}}}=0} + \ddot L_{1,2} \del_{L_{1,2}}\Phi = 0,$$

with $\;\dot {\mathsmaller{\sqrt{-D}}}>0.\;$ The conformal class of the candidate surfaces is given by 

$$(a, b) = c \big(  q_1 A_{1,2} + q_2L_{1,2},\; q_2 A_{1,2} - q_1 L_{1,2}\big) $$

since $\;b \equiv const, \;$  and $\;a\geq 0\;$ for $\;f^t\;$ we obtain

$$q_2 \ddot A_{1,2} -  q_1  \ddot L_{1,2} = 0,\quad \text{ and } \quad q_1\ddot A_{1,2} + q_2 \ddot  L_{1,2} >0.$$

From this we can conclude $\;\ddot L_{1,2} >0\;$ and hence 

$$\int_{\gamma_{1,2}} \ddot R\; ds  >0 \quad \text{and also} \quad \int_{\gamma_{1,2}} \big(\mathsmaller{\del^2_{\sqrt{-D} \sqrt{-D}}}R\big)\mathlarger{|}_{{\mathsmaller{\sqrt{-D}}}=0} \;ds >0.$$

Now, a straight forward computation shows

$$\big( \del^2_{\sqrt{-D} \sqrt{-D}} \Phi\big)\mathlarger{|}_{{\sqrt{-D}}=0}.$$

For $\;b \sim 1\;$ we have already computed   

$$\;\del_{\mu_\theta}\Phi\mathlarger{|}_{\mu_\theta=\mu_\theta(f^b)} <0\;\text{ and }\;\del_{L_{1,2}}\Phi\mathlarger{|}_{L_{1,2} = L_{1,2}(f^b)}> 0,\;$$ 

see Proposition \ref{delmu}. Therefore, we obtain 

$$\ddot \mu_\theta > 0.$$

 \end{proof}
 
 \appendix

\section{Stability properties of a penalized Willmore energy}\label{stability}
In the computations below we mostly follow \cite{KuwertLorenz} and thus we refer to that paper for details. To fix the notations, we consider immersions

$$f: T^2=\C / \Gamma \longrightarrow (S^3, \;g_{S^3}),$$

where $\;\Gamma\;$ is a lattice and $\;g_{S^3}\;$ is the round metric on $\;S^3.\;$ Let Imm$\;(\C / \Gamma)\;$ denote the space of all such immersions and $\;$Met$(\C / \Gamma)\;$ the space of all metrics on the torus $\;T^2.\;$ Moreover, let 

\begin{equation*}
\begin{split}G: \text{Imm }(T^2) &\longrightarrow  \text{Met }(T^2)\\
 f &\longmapsto f^*g_{S^3}
\end{split}
\end{equation*}

 be the map which assigns to every immersion its induced metric.  We denote by $\;\pi\;$ the projection from the space of metrics to the Teichm\"ulller space, which we model by the upper half plane $\;\H^2\;$ and with the notations above we can define $\;\bigPi\;$ to be:
 
 $$\bigPi = \pi \circ G : \text{Imm }(T^2) \longrightarrow \H^2.$$ 

As in \cite{KuwertLorenz} we parametrize the homogenous torus with conformal class $\;b= \tfrac{s}{r},\;$ and $\;r^2 + s^2 = 1\;$ as

\begin{equation}\label{parametrizationofhomogenous}
\begin{split}
 f^b: T^2_b = \C / \left(2\pi r \Z \oplus 2\pi i s \Z\right) &\longrightarrow S^3,\\
 (x,y) &\longmapsto\begin{pmatrix}r e^{i\tfrac{x}{r}}, &s e^{i\tfrac{y}{s}}\end{pmatrix}.
 \end{split}
\end{equation}

We want to compute the value of $\;\alpha^b\;$ which we recall to be

$$\alpha^b = \text{ max }\big\{\;\alpha \; | \  \delta^2 \mathcal W_{\alpha, \beta^b}\big(f^b\big) \geq0\;\big\}.$$
 
From \cite{KuwertLorenz} we can derive that $\;\alpha^b\;$ is characterized by the fact that $\;\delta^2 \mathcal W_{\alpha^b, \beta^b}|_{f^b}\geq 0\;$ and there exist a non-trivial normal variation $\;\varphi^b\;$ of $\;f^b\;$ such that 

$$ \delta^2 \mathcal W_{\alpha^b, \beta^b}\big(f^b\big)\big(\varphi^b, \varphi^b\big)= 0, \quad \text{and}\quad \delta^2 \mathcal W_{\alpha, \beta^b}\big(f^b\big)\big(\varphi^b, \varphi^b\big)<0, \text{ for } \alpha >\alpha^b.$$  

We will show that for $\;b \neq 1\;$ the variation $\;\varphi^b\;$ is unique up to scaling, isometry of the ambient space and reparametrization of the surface $\;f^b.\;$ We will also choose the orientation of $\;f^b\;$ and the variation $\;\varphi^b\;$ such that $\;\delta^2 \bigPi^1\big(f^b\big) \geq 0.$ \\

While for $\;b= 1\;$ the exact value of $\;\alpha^1\;$ and the associated normal variations can be computed, $\;\alpha^b\;$ for $\;b \neq 1\;$ does not have a nice explicit form. Nevertheless, we will show that the unique normal variation $\;\varphi^b\;$ characterizing $\;\alpha^b\;$ remain the same (in a appropriate sense) for all $\;b\sim1.\;$  In fact, the normal variation $\;{\displaystyle \lim_{a\rightarrow 0} }\big(\del_{\sqrt{a}} f_{(a,b)}\big)^\perp\;$ is the information we use to show that the Lagrange multipliers of the candidates $\;f_{(a,b)}\;$ converge to the $\;\alpha^b\;$ as $\;a \longrightarrow 0,\;$ see Theorem \ref{explicitcandidates}.\\

We first restrict to the case $\;b=1$ -- the Clifford torus. Since $\;\beta^1 = 0\;$ 
we investigate for which $\;\alpha\;$ the Clifford torus $f^1$ is stable for the penalized Willmore functional $\;\mathcal W_{\alpha} = \mathcal W - \alpha\bigPi^1.$

The second variation of the Willmore functional is well known. Thus we first concentrate on the computation of the second variation of $\;\bigPi^1.\;$ Another well known fact is $\;\delta \bigPi^1(f^1) = 0.\;$ Moreover, we have 

$$D^2 \bigPi^1\big(f^1\big)\big(\varphi, \varphi\big) = D^2\pi^1\Big(G\big(f^1\big)\Big)\Big(DG\big(f^1\big)\varphi, DG\big(f^1\big)\varphi\Big) +  D\pi^1\Big(G\big(f^1\big)\Big)\Big(D^2G\big(f^1\big)\big(\varphi,\varphi\big)\Big)$$

The first term is computed in Lemma $4$ of \cite{KuwertLorenz} to be

$$D \pi^1\Big(G\big(f^1\big)\Big)\Big(D^2G\big(f^1\big)\big(\varphi,\varphi\big)\Big) = - \tfrac{1}{\pi^2} \int_{T^2_1} <\nabla^2_{12} \varphi, \varphi> d \mu_{g_{euc}},$$

for normal variations $\varphi.$ It remains to compute the second term

$$D^2\pi^1\Big(G\big(f^1\big)\Big)\Big(DG\big(f^1\big)\varphi, DG\big(f^1\big)\varphi\Big).$$

By a straight forward computation (or by Lemma $2$ of \cite{KuwertLorenz}) we have  

$$D G\big(f^1\big)\varphi = - 2 \int_{T^2_1}<\text{II}, \varphi> d \mu_{g_{euc}},$$

where $\;\text{II}\;$ is the second fundamental form of the Clifford torus, which is trace free.\\

\noindent
Let $\;u\;$ and $\;v \in S_2(T^2_1)\;$ be symmetric $2-$forms satisfying 

$$\Tr_{euc} u = \Tr_{euc}v = 0 \quad \text{ and } \quad v \perp_{euc} S_2^{TT}(g_{euc}),$$

where $\;S_2^{TT}(g_{euc})\;$ is the space of symmetric, covariant, transverse traceless $2$-tensors with standard basis $\;q^1\;$ and $\;q^2\;$. Let $\;g(t) = g_{euc} + t u\;$ and $\;q^i(t) := q^i\big(g(t)\big)\;$ the corresponding basis with respect to $\;g(t)\;$ be as in \cite[Lemma 6]{KuwertLorenz} (see p.~10, l.~1). Then by \cite[Lemma 6]{KuwertLorenz} (p.~10, l.~2), we have that $\;\big(q^i(t) - q^i\big) \perp_{euc} S_2^{TT}(g_{euc}).\;$ On the other hand we can expand $\;v\;$ by 

$$v= v_i(t) q^i(t) + v^\perp(t), \quad \text{where }v^\perp(t)\perp_{g(t)} S_2^{TT}(g(t)).$$

By assumption we have $\;v_i(0) = 0\;$ and thus 
$$D^2\pi^1(g_{euc})(u, v) = \frac{d}{dt} D\pi^1\big(g(t)\big) \cdot v|_{t=0} = v'_{1}(0) D\pi^1(g_{euc}) \cdot q^1,$$
where
$$v'_{1}(0) = \tfrac{1}{4\pi^2} <v, (q^1)'(0)>_{L^2(g_{euc})},$$
as computed in  \cite{KuwertLorenz}. \\

Let $\;\eta:= \big(q^1\big)'(0) \; $ and $\;\eta^\circ = \eta_1 q^1 + \eta_2 q^2\;$ be its traceless part, then by the consequence of Lemma 6 of \cite{KuwertLorenz} given by the formula in page 11, line -12 in \cite{KuwertLorenz} (applied with \;$\mu=1$), and the formula in page 5, line 8 still in \cite{KuwertLorenz}, we have

\begin{equation}\label{etagleichung}
\begin{split}
(\text{div}_{euc} \eta^{\circ})_1 &= (\text{div}_{euc} u)_2\\
(\text{div}_{euc} \eta^{\circ})_2 &= (\text{div}_{euc} u)_1.
\end{split}
\end{equation}

For $\;u = u_1 q^1 + u_2 q^2\;$ we obtain,

$$ (\text{div}_{euc} u)_1 = \partial_2 u_1 - \partial_1 u_2, \quad (\text{div}_{euc} u)_2 = \partial_1 u_1 + \partial_2 u_2,$$

and therefore the Equations \eqref{etagleichung} become 

\begin{equation}\label{Ungleichung}
\begin{split}
 \partial_2 \eta_1 - \partial_1 \eta_2 &= \partial_1 u_1 + \partial_2 u_2\\
 \partial_1 \eta_1 + \partial_2 \eta_2 &= \partial_2 u_1 - \partial_1 u_2.
\end{split}
\end{equation} 

If we specialize to the relevant case $\;u = u_2 q^2\;$ and $\;v = v_2 q^2\;$ this yields 

$$\big(D^2\pi(g_{euc})(u, v)\big)_1 = \tfrac{1}{4\pi^2}< v_2 q^2 ,\eta^\circ>_{L^2(g_{euc})},$$ 

and we only need to concentrate on $\;\eta_2.\;$  Differentiating the Equations \eqref{Ungleichung} and subtracting these form each other gives (with $\;u_1 = 0$)

\begin{equation}\label{laplace}
\Delta \eta_2 = - 2 \partial_1\partial_2 u_2.
\end{equation}

In order to compute $\;\eta_2\; $ we restrict to normal variations $\;\varphi = \Phi \vec{n}\;$ for doubly periodic functions $\;\Phi\;$ in a Fourier space, i.e.,  $\;\Phi\;$ is a doubly periodic function on $\;\C\;$ with respect to the lattice \;$\;\sqrt{2}\pi \Z + \sqrt{2}\pi i \Z.\;$ The Fourier space $\;\mathcal F \big(T^2_1\big)\;$ of doubly periodic functions is the disjoint union of the constant functions and the \;$4$-dimensional spaces $\;\mathcal A_{kl}\big(T^2_1\big),\;$ $(k,l) \in \N\setminus \{(0,0)\}\;$ with basis

\begin{equation}
\begin{split}
&\sin\big(\sqrt{2} kx\big)\cos\big(\sqrt{2}ly\big), \quad \cos\big(\sqrt{2}kx\big)\sin\big(\sqrt{2}ly\big),\\ &\cos\big(\sqrt{2}kx\big)\cos\big(\sqrt{2}ly\big), \quad \sin\big(\sqrt{2}kx\big)\sin\big(\sqrt{2}ly\big).
\end{split}
\end{equation}

We restrict to the case where $\;\Phi = \Phi_{kl} \in \mathcal A_{kl},$ $(k,l) \in \N^2\setminus (0,0)\;$ in the following. Then for $\;u = v = \Phi_{kl} \vec{n}\;$ we obtain that 

$$\eta_2 = \tfrac{1}{k^2+ l^2}\partial_1\partial_2 \Phi_{kl}$$ 

solves equation \eqref{laplace}. The integration constant is hereby chosen such that $\;<\eta^{0}, q_1>_{L^2(g_{euc})} = 0.$
 
 Thus
 $$D^2\pi^1\big(G\big(f^1\big)\big)(u, v) = \tfrac{1}{2\pi^2 (k^2 + l^2)} \int_{T^2_1} \big(\partial^2_{12} \Phi_{kl}\big) \Phi_{kl}.$$
 
Put all calculations together we obtain

\begin{equation*}
\begin{split}
D^2\bigPi^1\big(f^1\big)\big(\varphi, \varphi\big)
&= -\tfrac{1}{\pi^2} \int_{T^2_1} \big(\partial^2_{12} \Phi_{kl}\big) \Phi_{kl}  \\
 &+ \tfrac{2}{\pi^2(k^2 + l^2)} \int_{T^2_1} \big(\partial^2_{12} \Phi_{kl}\big) \Phi_{kl}.
\end{split}
\end{equation*}

\begin{Rem}
The second variation for general normal variation $\;\varphi = \left(\sum_{k,l \in \N^2} a_{k,l} \Phi_{k,l}\right) \vec{n}\;$ is obtained by linearity. Terms obtain by pairing $\;\Phi_{k,l}\;$ and $\;\Phi_{m,n},\;$ where $\;(k,l) \neq (m,n)\;$ vanishes. To determine stability of $\;\mathcal W_\alpha\;$ we can thus restrict ourselves with out loss of generality to the case $\;\varphi = \Phi_{k,l} \vec{n}.$\\
\end{Rem}

Clearly, if for a normal variation $\;\varphi\;$ we have 

$$D^2\bigPi^1\big(f^1\big)\big(\varphi, \varphi\big) \leq 0,$$

 then by the stability of the Clifford torus 
 
$$D^2 \mathcal W_{\alpha}\big(f^1\big)\big(\varphi, \varphi\big) \geq 0$$  

for all $\;\alpha \geq 0.\;$ Moreover, for

\begin{equation*}
\begin{split}
\Phi_{kl} &= a \sin\big(\sqrt{2}kx\big)\cos\big(\sqrt{2}ly\big) + b \cos\big(\sqrt{2}kx\big)\sin\big(\sqrt{2}ly\big)\\ &+ c \cos\big(\sqrt{2}kx\big)\cos\big(\sqrt{2}ly\big) + d \sin\big(\sqrt{2}kx\big)\sin\big(\sqrt{2}ly\big)
\end{split}
\end{equation*}

with $\;k,l \in \N \setminus \{0\}\;$ and $\;a,\; b, \;c, \;d \in \R\;$ we have:

\begin{equation}\label{D^2Pi^1}
\begin{split}
D^2\bigPi^1\big(f^1\big)\big(\varphi, \varphi)  &=\tfrac{1}{\pi^2} \big(2kl - \tfrac{4kl}{k^2+l^2} \big)\tfrac{2ab - 2cd}{a^2 + b^2+ c^2+ d^2}<\varphi, \varphi>_{L^2}\\& \leq \tfrac{1}{\pi^2} \big(2kl - \tfrac{4kl}{k^2+l^2} \big)<\varphi, \varphi>_{L^2},
\end{split}
\end{equation}

with equality if and only if 
\begin{equation}\label{Pi2=}
a = b \quad \text{ and } \quad c = - d.
\end{equation}

\noindent
The second variation of the Willmore functional at the Clifford torus (Lemma 3 \cite{KuwertLorenz}) is given by:

\begin{equation}\label{D^2W}
\begin{split}
D^2 \mathcal W\big(f^1\big) (\varphi, \varphi) &= <\big(\tfrac{1}{2}\Delta^2 + 3\Delta + 4\big)\varphi,\varphi >_{L^2}\\
&= \big(2 (k^2 + l^2)^2 - 6(k^2 + l^2) + 4\big)<\varphi,\varphi>_{L^2}.
\end{split}
\end{equation}

Therefore we have

$$D^2 \mathcal W\big(f^1\big) (\varphi, \varphi) =  0,$$ if and only if
$\;k =\pm 1\;$ and $\;l = \pm 1,\;$
or $\;k = 0\;$ and $\;l = \pm1,\;$
 or $\;k = \pm1\;$ and $\;l = 0.$\\

\noindent
Let $\;c:= \tfrac{k}{l}\;$ and we assume without loss of generality that $\;c \geq 1,\;$ then the second variation formulas \eqref{D^2Pi^1} and \eqref{D^2W} simplifies to:

\begin{equation*}
\begin{split}
D^2 \mathcal W\big(f^1\big) (\varphi, \varphi) &= \big(2 (c^2 + 1)^2 l^4 - 6 (c^2 + 1) l^2 + 4\big)<\varphi,\varphi>_{L^2}\\
D^2\bigPi^1\big(f^1\big)(\varphi, \varphi)  &\leq  \tfrac{1}{\pi^2} \big(2c l^2 -4 \tfrac{c}{c^2 +1}\big)<\varphi,\varphi>_{L^2}.
\end{split}
\end{equation*}

Hence we obtain for $\; \tilde\alpha = \tfrac{1}{4\pi^2} \alpha$

$$D^2 \mathcal W_{\alpha} (f^1)(\varphi, \varphi) \geq \Big(2 (c^2 + 1)^2 l^4 - \big(6(c^2 + 1) +  8\tilde \alpha c\big)l^2 + 4 + 16\tilde \alpha \tfrac{c}{c^2 +1} \Big) <\varphi,\varphi>_{L^2}.$$

with equality is and only if $\;\Phi\;$ is satisfies \eqref{Pi2=}.
We still want to determine the range of $\;\alpha\;$ for which $\;\mathcal W_{\alpha}\;$ is stable.  At $\;\alpha = \alpha^b\;$ the second variation of $\;\mathcal W_{\alpha}\;$ have zero directions in the normal part which are not  M\"obius variations. Thus we need to determine the roots of the polynomial 

$$g_{\tilde \alpha, c}(l): = \Big(2 (c^2 + 1)^2 l^4 - \big(6(c^2 + 1) +  8\tilde \alpha c\big)l^2 + 4 + 16\tilde \alpha \tfrac{c}{c^2 +1} \Big)$$

The polynomial $\;g_{\tilde \alpha, c}\;$ is even, its leading coefficient is positive and its roots satisfy:

\begin{equation}\label{rootsl}
l^2 = \tfrac{2}{c^2 +1}, \text{ or } \quad l^2= \tfrac{1}{c^2 +1} + 4 \tilde \alpha\tfrac{c}{(c^2 +1)^2}.
\end{equation}

The values of $\;l \in \N\;$ for which $\;g_{\tilde \alpha, c}\;$ is negative lies exactly between the positive roots of $\;g_{\tilde \alpha, c}.\;$ So we want to determine $\;\tilde \alpha\;$ such that this region of negativity  for $\;g_{\tilde \alpha, c},\;$ i.e., the interval between the two positive solutions $\;l_1(\tilde \alpha, c)\;$ and $\;l_2(\tilde \alpha, c)\;$ of \eqref{rootsl} contain no positive integer for all $\;c  = \tfrac{k}{l}\;$ (other than those combinations leading to a M\"obius variation).
\noindent 
We consider two different cases:

$$c= 1 \text{ and } c > 1.$$

For $\;c = 1\;$ the four roots of $\;g_{\tilde \alpha,1}\;$ are determined by:  

$$l^2 =1, \quad l^2 = \tfrac{1}{2} + \tilde \alpha.$$

Since the case of $\;l^2 = 1,\;$ i.e., $\;l=k = \pm1\;$ corresponds to M\"obius variations, we can rule out the existence of negative values of $\;g_{\tilde \alpha, 1}\;$ if and only if the second root satisfies 

$$|l| \leq 2, \quad \text{ or equivalently, } \quad  l^2 = \tfrac{1}{2} + \tilde \alpha \leq 4.$$

 From which we obtain $\;\tilde \alpha \leq \frac{7}{2}.$\\

For $\;c > 1,\;$ the first equation $\;l^2 = \tfrac{2}{c^2 +1}< 1\;$ is never satisfied for an integer $\;l.\;$ Thus we only need to consider the equation
$$l^2 = \tfrac{1}{c^2 +1} + 4 \tilde \alpha\tfrac{c}{(c^2 +1)^2}.$$

To rule out negative directions for $\;D^2\mathcal W_{4 \pi^2 \tilde \alpha}\;$ it is necessary and sufficient to have 

$$l^2 = \tfrac{1}{c^2 +1} + 4 \tilde \alpha\tfrac{c}{(c^2 +1)^2} \leq 1 $$

for appropriate $\;c= \tfrac{k}{l}.\;$ For $\;l^2 = 1\;$ we obtain that $\;c = \tfrac{k}{l} \in \N_{>1}\;$ and $\;\tilde \alpha\;$ satisfies:

$$\tilde \alpha  =  \tfrac{1}{4} (c^3+ c).$$

The right hand side is monotonic in $\;c\;$ and therefore the minimum for $\;c \in \N_{>1}\;$ is attained at $\;c=2\;$ which is equivalent to $\;\tilde \alpha  =\frac{5}{2}.\;$ Since $\;\frac{5}{2} < \tfrac{7}{2}\;$ which was the maximum $\;\tilde \alpha\;$ in the $\;c=1\;$ case, we get that  $\;\delta^2\mathcal W_{10 \pi^2}\big(f^1\big) \geq 0.\;$ Further, at $\;\tilde \alpha = \frac{5}{2}\;$ the (non-M\"obius) normal variations in the kernel of $\;\delta^2\mathcal W_{10 \pi^2}\big(f^1\big)\;$ are given by

\begin{equation}
\begin{split}
\Phi_1 &= \sin(2 \sqrt{2} y) \cos(\sqrt{2} x) +  \cos(2\sqrt{2} y) \sin(\sqrt{2} x) = \sin\big(\sqrt{2}(x+ 2y)\big)\\
\tilde  \Phi_1  &= \sin(2 \sqrt{2} y) \sin(\sqrt{2} x) - \cos(2\sqrt{2} y) \cos(\sqrt{2} x) = \cos\big(\sqrt{2}(x+2y)\big)
\end{split}
\end{equation}

and by symmetry of $\;k\;$ and $\;l\;$ (we have assumed $\;c\geq1$): 

\begin{equation}
\begin{split}
\Phi_2 &= \sin(2 \sqrt{2} x) \cos(\sqrt{2} y) +  \cos(2\sqrt{2} x) \sin(\sqrt{2} y) = \sin\big(\sqrt{2}(2x + y)\big)\\
\tilde  \Phi_2  &= \sin(2 \sqrt{2} x) \sin(\sqrt{2} y) - \cos(2\sqrt{2} x) \cos(\sqrt{2} y)= \cos\big(\sqrt{2}(2x+y)\big),
\end{split}
\end{equation}

where $\;\tilde \Phi_i (x,y) = \Phi_i (x, y+ \tfrac{\pi}{2}),\;$ i.e., $\;\Phi_i$ and $\tilde \Phi_i\;$ differ only by a translation. We have shown the following Lemma.\\

\begin{Lem}\label{kernvonalpha}
At $\;b = 1\;$ we have that 

$$\alpha^1 = \text{ max }\big\{\;\alpha >0 \  | \ \delta^2\mathcal W_{\alpha}\big(f^1\big) \geq 0\; \big\}$$

is computed to be $\;10 \pi^2.\;$\\
\end{Lem}
The problem at $\;b= 1\;$ is that the kernel dimension of $\;\delta^2 \mathcal W_{\alpha^1}\big(f^1\big)\;$ is too high. Even using the invariance of the equation it is not possible to reduce it to $1$, which is needed for the bifurcation theory from simple eigenvalues. The main reason is that linear combinations of the two $\;\Phi_i\;$ cannot be reduced to a translation and scaling of $\;\Phi_1\;$ only. This situation is different for $\;b\neq 1,\;$ see Proposition \ref{1d}, because for homogenous tori \eqref{parametrizationofhomogenous} the immersion is not symmetric w.r.t. parameter directions $\;x\;$ and $\;y\;$. For $\;b \neq 1\;$ we have that $\;\beta^b \neq 0\;$ and thus the second variation of $\;\bigPi^2\;$ enters the calculation of  
$$\alpha^b= \text{ max }\big\{\;\alpha>0 \;| \; \delta^2\mathcal W_{\alpha, \beta^b} \geq0\;\big\}.$$

Moreover, $\;A_{k,l}\big( T^2_1\big)\;$ is canonically isomorphic to $\;A_{k,l}\big( T^2_b\big)\;$ via

\begin{equation}
\begin{split}
\sin\big(k \sqrt{2}x\big)\cos\big(l \sqrt{2}y\big) &\longmapsto \sin\big(\tfrac{k x}{r}\big)\cos\big(\tfrac{l y}{s}\big),\\
\sin\big(k \sqrt{2}x\big)\sin\big(l \sqrt{2}y\big) &\longmapsto \sin\big(\tfrac{k x}{r}\big)\sin\big(\tfrac{l y}{s}\big),\\
\cos\big(k \sqrt{2}x\big)\sin\big(l \sqrt{2}y\big) &\longmapsto \cos\big(\tfrac{k x}{r}\big)\sin\big(\tfrac{l y}{s}\big),\\
\cos\big(k \sqrt{2}x\big)\cos\big(l \sqrt{2}y\big) &\longmapsto \cos\big(\tfrac{k x}{r}\big)\cos\big(\tfrac{l y}{s}\big).
\end{split}
\end{equation}

To emphasis this isomorphism, we denote in the following normal variations at the Clifford torus by $\;\varphi^1 = \Phi^1 \vec{n}^1,\;$ with $\;\Phi^1\;$ a well defined function on $\;T^2_1,\;$  and the corresponding normal variations at homogenous tori $f^b$ under the above isomorphism by $\;\varphi^b = \Phi^b \vec{n}^b.\;$
Since $\;\delta^2 \mathcal W \geq0\;$ and $\;\delta^2\bigPi^1 \geq0\;$
we obtain the following Lemma using Lemma 4 and 7 of \cite{KuwertLorenz}.\\

\begin{Lem}\label{spectralproperty}
With the notations as above let $\;\alpha_0 \in \R_+\;$ be fixed and $\;\Phi^1 \in A_{k,l}(T^2_1)\;$ such that 

$$\delta^2\mathcal W_{\alpha_0}\big(f^1\big)\big( \varphi^1,  \varphi^1\big) >0.$$

Then for $\;b\sim1\;$ close enough we also have

$$\delta^2\mathcal W_{\alpha, \beta^b}\big(f^b\big)\big(  \varphi^b, \varphi^b\big)>0,$$

for all $\;\alpha \leq \alpha_0.\;$\\
\end{Lem}

Kuwert and Lorenz  \cite{KuwertLorenz} computed the second derivative of $\;\bigPi^2\;$ for $\;\varphi^b = \Phi^b_{k,l} \vec{n}^b \;$ to be 

\begin{equation}
\begin{split}
D^2 \bigPi^2\big(f^b\big)\big( \varphi^b,  \varphi^b\big) &=\tfrac{1}{4\pi^2 r^2} \int_{T^2_b}<\del^2_{11} \Phi^b- \del^2_{22} \varphi^b,  \varphi^b>dA\\
 &+ \tfrac{r^2-s^2}{4\pi^2 r^4s^2}\int_{T^2_b} | \varphi^b|^2dA \\
&- \tfrac{2(r^2-s^2) + c_r(k,l)}{4\pi^2 r^4s^2} \int_{T^2_b} | \varphi^b|^2dA,
\end{split}
\end{equation}

where $\; c_r(k,l) := \tfrac{k^2s^2- l^2r^2}{k^2s^2+l^2r^2} \;$ and $\; \varphi^b \in A_{k,l}\big(T^2_b\big) \vec{n}^b.$

For $\;b \sim 1,\;$ i.e., $\;r \sim \tfrac{1}{\sqrt{2}}\;$ this yields 

$$D^2\bigPi^2\big(f^b\big)\big( \varphi^b_1,  \varphi^b_1\big) >D^2 \bigPi^2\big(f^b\big)\big( \varphi^b_2,  \varphi^b_2\big),$$

for $\; \varphi_i^b\;$ are the images of $ \;\varphi_i^1 \in $ Ker $\delta^2\mathcal W_{\alpha^1}\big(f^1\big)\;$ under the canonical isomorphism and since for  $\;b>1,\;$ i.e.,  $\;r<s\;$ and $\;\beta^b >0\;$ we obtain

$$\delta^2 \mathcal W_{\alpha^1, \beta^b}\big(f^b\big)\big( \varphi^b_1,  \varphi^b_1\big) < \delta^2 \mathcal W_{\alpha^1, \beta^b}\big(f^b\big)\big( \varphi^b_2,  \varphi^b_2\big) < 0.$$

Thus $\;\alpha^b < \alpha^1\;$ and we obtain that for $\;b \sim 1\;$ and $\;b \neq 1\;$ the kernel of $\;\delta^2\mathcal W_{\alpha^b, \beta^b}\big(f^b\big)\;$ is $2$-dimensional and consists of either $\; \varphi^b_1\;$ and $\;\tilde  \varphi^b_1\;$ for $\;b>1\;$  or $ \;\varphi^b_2\;$ and $\; \varphi^b_2\;$ for $\;b<1.\;$ Both choices of $\;b\;$ lead to M\"obius invariant surfaces. We summarize the results in the following Lemma:\\

\begin{Lem}\label{stabilitykernel}
For $\;b \sim1,\;$ $\;b>1\;$ we have that $\;\alpha^b\;$ is uniquely determined by the kernel of 
$\;\delta^2 \mathcal W_{\alpha^b, \beta^b}\big(f^b\big)\;$ which is $2$ dimensional and spanned  (up to invariance) by the normal variations 

$$ \varphi^b_1=  \sin\big(\tfrac{x}{r} + \tfrac{2y}{s}\big) \vec{n}^b\quad \text{and} \quad \tilde \varphi^b_1 =  \cos\big(\tfrac{x}{r} + \tfrac{2y}{s}\big)\vec{n}^b.$$

\end{Lem}
Now, for $\;b\sim 1\;$ consider the reparametrization of the homogenous torus as a $(2,-1)$-equivariant surface

\begin{equation*}
\begin{split}
\tilde f^b: \C /\big(2\pi i\Z + 2\pi \tfrac{sr + 2ir^2}{4r^2 + s^2} \Z\big) & \longrightarrow S^3\subset \C^2, \\
 (\tilde x, \tilde y)&\longmapsto \begin{pmatrix}r e^{i2\tilde y +  \tfrac{i s \tilde x}{r}}, s e^{-i\tilde y+  \tfrac{   i  2r \tilde x}{s}}\end{pmatrix}.
 \end{split}
 \end{equation*}

Using these new coordinates the kernel of $\;\delta^2 \mathcal W_{\alpha^b, \beta^b}\big(f^b\big)\;$ for $\;b = \tfrac{s}{r}>1\;$ is given by

$$\Phi^b= \sin \big((\tfrac{s}{r}+ 4 \tfrac{r}{s})\tilde x\big), \quad \tilde \Phi^b = \cos\big((\tfrac{s}{r}+ 4 \tfrac{r}{s})\tilde x\big).$$

Thus infinitesimally the $\;\tilde y$-direction of the surface is not affected by a deformation with normal variation $\;\Phi^b\vec{n}^b,\;$ i.e., the $\;(2, -1)$-equivariance is infinitesimally preserved.

This is the most natural coordinate change with respect to these stability computations and we have

$$\ \begin{pmatrix} x \\ y \end{pmatrix}  \longmapsto \begin{pmatrix} s& 2r\\2r& -s\end{pmatrix}  \begin{pmatrix} \tilde x \\ \tilde y\end{pmatrix} $$

Since the space of $\;(2,-1)$-equivariant surfaces and $\;(1,2)$-equivariant surfaces are isomorphic and differ only by the orientation of the surface and an isometry of $\;S^3\;$, we consider $\;(1,2)$-equivariant surfaces for convenience. 
Moreover, it is important to note that for all real numbers $\;c_1,\;$ $c_2\;$ there exist $\;d_1, d_2 \in \R\;$ such that

\begin{equation}\label{additionstheoreme}
\begin{split}
c_1 \Phi_1^b + c_2 \tilde \Phi_1^b &= c_1 \sin\big((\tfrac{s}{r}+ 4 \tfrac{r}{s})\tilde x\big) + c_2 \cos\big((\tfrac{s}{r}+ 4 \tfrac{r}{s})\tilde x\big) \\
&= d_1 \sin\big((\tfrac{s}{r}+ 4 \tfrac{r}{s})\tilde x+ d_2\big) =  d_1\Phi_1^b\big((\tfrac{s}{r}+ 4 \tfrac{r}{s})\tilde x + d_2\big).
\end{split}
\end{equation}


\begin{thebibliography}{1000000000}



\bibitem[AmbPro]{AP} {A.~Ambrosetti, G.~Prodi.} { A Primer of Nonlinear Analysis,} Textbook, {\em Cambridge University Press,} Cambridge, 1995. 

\bibitem[BauKuw]{BauerKuwert} {M.~Bauer, E.~Kuwert}.
        { Existence of minimizing Willmore surfaces of prescribed genus,}{  {\em Int. Math. Res. Not., }
         {(2003)},
        {pp 553--576}.}


\bibitem[BerRiv]{BernardRiviere} {Y.~Bernard, T.~Rivi`ere}. {Energy quantization for Willmore surfaces and application}
       { {\em Annals of Math.}, {\bf 180} (2014), no. 1, pp 87--136.}
               
\bibitem[Bla]{Blaschke}  W.~Blaschke. {Vorlesungen \"uber Differentialgeometrie. Vol. 3}, Springer-Verlag, Berlin, 1929.


  
\bibitem[Bob]{B}  A.~I.Bobenko. { All constant mean curvature tori in $\R^3$, $\S^3$, $\H^3$ in terms of theta-functions,} {\em Math. Ann.,} {\bf 290} (1991), no. 2, pp 209--245.

  
\bibitem[BoPePi]{BohPedPin_ana} C.~Bohle, F.~Pedit, U.~Pinkall. The spectral
  curve of a quaternionic holomorphic line bundle over a
  2-torus, \emph{Manuscripta Math.,}  {\bf 130} (2009), pp 311--352.

\bibitem[Boh]{Bohle} {C.~Bohle}. {Constrained Willmore tori in the 4-sphere,}{  {\em Journal of Differential Geometry}, {\bf {86}} {(2010)},  {pp 71--131}.}

  
\bibitem[BoLePePi]{BohLesPedPin} C.~Bohle, K.~Leschke, F.~Pedit,
  U.~Pinkall. Conformal maps of a 2-torus into the
  4-sphere, \emph{J.~Reine Angew.\ Math.},  {\bf 671} (2012), pp 1--30.
  


\bibitem[Bre]{Br} S.~Brendle. {\em Embedded minimal tori in $S^3$ and the Lawson conjecture},  {\em Acta Math.,}  {\bf 211} (2013), no. 2, 177--190.

\bibitem[Bry]{Bryant} {R.~Bryant}.
        {A duality theorem for Willmore surfaces,}{  {\em Journal of Differential Geometry},
         {\bf {20}} {(1984)},
        {pp 23--53}.}


\bibitem[BuPePi]{BuPP} F.~Burstall, F.~Pedit and U.~Pinkall. {Schwarzian Derivatives and Flows of Surfaces}, {\em Contemp. Math.}, {\bf 308} (2002), pp 39--61.
 
 
\bibitem[Ch]{Chen} B.~Y.~Chen. {Some conformal invariants of submanifolds and
their applications}, {\em Bollettino dell Unione Matematica Italiana}, {\bf 10} (1974), no. 4, pp. 380--385. 

 
 \bibitem[DaFrGrSc]{Grunau} A.~Dall'Aqua, S.~Fr\"ohlich, H.-C~Grunau, F.~Schieweck. Symmetric Willmore surfaces of revolution satisfying arbitrary Dirichlet boundary data. {\em Adv. Calc. Var.,} {\bf 4} (2011), no. 1, pp. 1-81.
  
\bibitem[ErMaOb]{Zentgraf} A.~Erdelyi, W.~Magnus, F.~Oberhettinger, F.~G.~Tricomi. Higher transcendental functions Vol II, {\em McGraw-Hill Book Company Inc.}, New York, 1953.

\bibitem[FeLePePi]{FLPP}  D.~Ferus, K.~Leschke, F.~Pedit, and U.~Pinkall. {\em Quaternionic Holomorphic Geometry: Pl\"ucker Formula, Dirac Eigenvalue Estimates and Energy Estimates of Harmonic 2-Tori}, {\em Invent.  Math.}, {\bf{146}} (2001),  no. 3, pp 507 -- 593.

 \bibitem[FerPed]{FerusPedit} {D.~Ferus, F.~Pedit}.  {$S^1$-equivariant minimal tori in $S^4$ and $S^1$-equivariant Willmore tori in $S^3$,}{  {\em Math. Z., } {\bf {204}} {(1990)}, {no. 2}, {269--282}.}
 

  \bibitem[Hel]{Hethesis}L.~Heller. {Equivariant constrained Willmore tori in the $3$-sphere,} Doctoral thesis at the University of T\"ubingen, 2012.
  
  
\bibitem[Hel1]{He1} L.~Heller.  {Equivariant Constrained Willmore Tori in the $3$-sphere}. {\em Math. Z.,} {\bf 278} (2014), no. 3, pp 955--977.


 \bibitem[Hel2]{He2} L.~Heller. {Constrained Willmore tori and elastic curves in $2$-dimensional space forms}, {\em Comm. Anal. Geom.}, {\bf 22}  (2014), no. 2, pp 343--369.
 
 \bibitem[HelNdi]{HelNdi1} L.~Heller, Ch.~B.~Ndiaye. {First explicit constrained Willmore minimizers of non-rectangular conformal class}. Preprint: arXiv:1710.00533.
 
  
 \bibitem[HelPed]{HePe} L.~Heller, F.~Pedit. {Towards a constrained Willmore conjecture}, in: Willmore Energy and Willmore Conjecture,  Edited by M.~Toda, Taylor \& Francis Inc, 2017. Preprint: arXiv:1705.03217. 
 
 \bibitem[Hit]{H}  N.~J.~Hitchin. {Harmonic maps from a $2$-torus to the $3$-sphere}, {\em Journal of Differential Geometry,} {\bf 31} (1990), no. 3,  pp 627--710. 
 
 
 \bibitem[KiScSc1]{KilianSchmidtSchmitt1} M.~Kilian, M.~U.~Schmidt, N.~Schmitt. {Flows of constant mean curvature tori in the 3-sphere: the equivariant case,} {\em J. Reine Angew. Math.,} {\bf 707 }(2015), pp 45--86.
 
  \bibitem[KiScSc2]{KilianSchmidtSchmitt2} M.~Kilian, M.~U.~Schmidt, N.~Schmitt.  {On stability of equivariant minimal tori in the 3-sphere,} {\em J. Geom. Phys.}, {\bf 85} (2014), pp 171--176.
  
  \bibitem[KoeKri]{KoecherKrieg} M.~Koecher, A.~Krieg. {Elliptische Funktionen und Modulformen}, Springer-Verlag Berlin Heidelberg, 2007.
  
  \bibitem[KuwLi]{KuwertLi}E.~Kuwert, Y.~Li, $W^{2,2}$-conformal immersions of a closed Riemann surface into $\R^n$, {\em Comm. Anal. Geom.}, {\bf 20} (2012), no.2, pp 313--340.
  
  
 \bibitem[KuwLor]{KuwertLorenz}E.~Kuwert, J.~Lorenz. On the stability of the CMC Clifford tori as constrained Willmore surfaces, {\em Ann. Global Anal. Geom.}, {\bf44} (2013), no. 1, pp 23--42.
  
      \bibitem[KuwSch]{KuwertSchaetzle2}{E.~Kuwert, R.~Sch\"atzle.}  {The Willmore Functional}, {\em Topics in Modern regularity theory}, pp 1--115, CRM Series, 13, Ed. Norm., Pisa, 2012.  
            
    \bibitem[KuwSch1]{KuwertSchaetzle2}{E.~Kuwert, R.~Sch\"atzle.}  {Closed surfaces with bounds on their Willmore energy}, {\em Annali della Scuola Normale Superiore di Pisa - Classe di Scienze}, {\bf {11}} {(2012)}, { pp 605--634}.
    
  \bibitem[KuwSch2]{KuwertSchaetzle}{E.~Kuwert, R.~Sch\"atzle.}  {Minimizers of the Willmore functional under fixed conformal class}, {\em Journal of Differential Geometry}, {\bf {93}} {(2013)}, { pp 471--530}.

 \bibitem[LiYau]{LiYau} {P.~Li, S.T.~Yau}.
        {A new conformal invariant and its applications to the Willmore conjecture and the first eigenvalue of compact surfaces,}{  {\em Invent. Math.},
        {\bf {69} (1982)},
        { pp 269--291}.}
 
 \bibitem[MarNev]{MarquesNeves} {F.~Marques, A.~Neves}.
        {Min-Max theory and the Willmore conjecture, }{\em Annals of Math.}, {{\bf 179} (2014), pp 683--782.}

\bibitem[MonRiv]{Mondino} A.~Mondino, T.~Rivi\`ere. Willmore Spheres in compact Riemannian manifolds. {Adv. Math.}, {\bf 232} (2013), pp 608--676.
 
 \bibitem[MonRos]{MontielRos} {S.~Montiel, A.~Ros}.
        {Minimal immersion of surfaces by the first eigenfunction and conformal area,}{\em Invent. Math.}, {{\bf 83} (1986), pp 153--166.}


 \bibitem[NdiSch1]{NdiayeSchaetzle1}C.~B.~Ndiaye, R.~M.~Sch\"atzle. {New examples of conformally constrained Willmore minimizers of explicit type,} {\em Adv. Calc. Var.,} {\bf 8} (2015), no. 4, pp 291--319.
 
 \bibitem[NdiSch2]{NdiayeSchaetzle2}C. B. Ndiaye, R. M. Sch\"atzle. { Explicit conformally constrained Willmore minimizers in arbitrary codimension,} {\em Calc. Var. Partial Differential Equations}, {\bf 51} (2014), no. 1-2, pp 291--314.
  
\bibitem[Pin]{Pinkall} U.~Pinkall. { Hopf Tori in $S^3$,} 
{\em Invent. Math.}, {\bf 81} (1985), no. 2, pp 379--386.


\bibitem[Riv]{Riviere} {T.~Rivi\`ere.}{Analysis aspects of the Willmore functional,}{  {\em  Invent. Math.}, {\bf {174}} {(2008)}, {no. 1}, pp 1--45.}

\bibitem[Riv2]{Riviere2} {T.~Rivi\`ere.} {Variational Principles for immersed surfaces with $L^2-$ bounded second fundamental form,}{  {\em J. reine angw. Math. }, {\bf {695}}, {(2014)}, pp 41--98.}

\bibitem[Riv3]{Riviere3} {T.~Rivi\`ere.} {Weak immersions of surfaces with $L^2$-bounded second fundamental form.} {\em Geometric Analysis}, pp 303--384, IAS/Park City Math. Ser., 22, Amer. Math. Soc., Providence, RI, 2016.


\bibitem[Ros]{Ros} A.~Ros. {The Willmore Conjecture in the Real Projective Space}, {\em Math. Research Letters},  {\bf {6}} (1999), pp 487--493.
 
 
 \bibitem[Sch]{Schaetzle} {R.~Sch\"atzle.} 
{ Conformally constrained Willmore immersions,} {\em Adv. Calc. Var.}, {\bf 6} (2013), no. 4, pp 375--390.
 
  \bibitem[Schm]{Schmidt} {M.~U.~Schmidt.}  { A proof of the Willmore conjecture,} Preprint: arXiv:math/0203224, 2002.

        
 \bibitem[Sim]{Simon} {L.~Simon}. {Existence of Surfaces minimizing the Willmore Functional,}{  {\em Commun. Anal. and Geom.},
        {\bf {1}} {(1993)},
        {pp 281--326}.}
        
\bibitem[Top]{Top}{P.~Topping}. {Towards the Willmore conjecture,} {\em Calc. Var.},  {\bf 11} (2000), pp 361--393.

\bibitem[Wei]{Weiner} {J.~Weiner}. 
        { On a problem of Chen, Willmore, et al.,} 
        { {\em Indiana Univ. Math. J.},{ \bf {27}} (1978), pp 19--35.}
        
\bibitem[Wil]{Willmore} {T.~Willmore}.
        {Note on embedded surfaces,} 
        { {\em An. Stiint. Univ. ``Al. I. Cuza" Iasi Sect. I a Mat.},{ \bf {11}}, (1965), pp 493--496.}

\end{thebibliography}
\end{document}